\documentclass[11pt,a4paper]{article}
\usepackage{amssymb,amsmath,amsthm,epsfig,verbatim,pstricks,url}
\usepackage{graphbox} 
\usepackage{enumitem} 
\usepackage{ifpdf}  
\newtheorem{theorem}{\sc Theorem.}[section]
\newtheorem{lemma}[theorem]{\sc Lemma.}
\newtheorem{remark}[theorem]{\sc Remark.}

\newtheorem{definition}[theorem]{\sc Definition.}

\newenvironment{AMS}%
{{\upshape\bfseries AMS subject classifications. }\ignorespaces}{}
\newenvironment{keywords}{{\upshape\bfseries Key words. }\ignorespaces}{}

\newcommand{\bR}{{\mathbb R}}

\newcommand{\sigmaO}{o}
\newcommand{\GT}{{\mathcal{G}_T}}           
\newcommand{\GhT}{{\mathcal{G}^h_T}}
\newcommand{\Whh}{V(\Gamma^h)}
\newcommand{\Wht}{V(\Gamma^h(t))}
\newcommand{\Vht}{\underline{V}(\Gamma^h(t))}
\newcommand{\Whm}{V(\Gamma^m)}
\newcommand{\Vhh}{\underline{V}(\Gamma^h)}
\newcommand{\Vhm}{\underline{V}(\Gamma^m)}
\newcommand{\WhGhT}{V(\GhT)}
\newcommand{\VhGhT}{\underline{V}(\GhT)}
\newcommand{\nabs}{\nabla_{\!s}}
\newcommand{\tG}{{\widetilde{G}}}
\newcommand{\vol}{\operatorname{vol}}
\newcommand{\diag}{\operatorname{diag}}
\newcommand{\diam}{\operatorname{diam}}
\newcommand{\dist}{\operatorname{dist}}

\newcommand{\dH}[1]{\;{\rm d}{\mathcal{H}}^{#1}} 
\newcommand{\dL}[1]{\;{\rm d}{\mathcal{L}}^{#1}} 

\newcommand{\id}{{\rm id}}

\newcommand{\dd}[1]{\frac{\rm d}{{\rm d}#1}}
\newcommand{\ddt}{\dd{t}}
\newcommand{\dt}{\;{\rm d}t}
\def\epsilon{\varepsilon} 
\newcommand{\ttau}{\Delta t}

\def\hat{\widehat}

\newcommand{\errorXx}{\|\Gamma^h - \Gamma\|_{L^\infty}}
\newcommand{\errorUu}{\|U^h - u\|_{L^\infty}}
\begin{document}
\title{
A structure preserving front tracking finite element method for the 
Mullins--Sekerka problem}
\author{Robert N\"urnberg\footnotemark[3]}

\renewcommand{\thefootnote}{\fnsymbol{footnote}}
\footnotetext[2]{Dipartimento di Mathematica, Universit\`a di Trento,
38123 Trento, Italy \\ {\tt robert.nurnberg@unitn.it}}

\date{}

\maketitle

\begin{abstract}
We introduce and analyse a fully discrete approximation for a mathematical
model for the solidification and liquidation of materials of 
negligible specific heat. The model is a two-sided Mullins--Sekerka problem.
The discretization uses finite elements in space and an independent
parameterization of the moving free boundary. We prove unconditional stability
and exact volume conservation for the introduced scheme.
Several numerical simulations, including for nearly crystalline surface
energies, demonstrate the practicality and accuracy of the presented numerical
method.
\end{abstract} 

\begin{keywords} 
Mullins--Sekerka, Hele--Shaw, anisotropy, parametric finite element method,
unconditional stability, volume preservation
\end{keywords}

\begin{AMS} 
35K55, 35R35, 65M12, 65M50, 65M60, 
74E10, 74E15, 80A22
\end{AMS}
\renewcommand{\thefootnote}{\arabic{footnote}}

\setcounter{equation}{0}
\section{Introduction} 

In this paper we propose and analyse a novel numerical approximation of the
following moving boundary problem. Let $\Omega\subset\bR^d$, $d\geq2$, 
be a domain with a Lipschitz boundary $\partial\Omega$ and
outer unit normal $\vec\nu_{\Omega}$.
Given the hypersurface $\Gamma(0) \subset \Omega$, 
find $u:\Omega\times [0,T]\to\bR$ and the evolving hypersurface
$(\Gamma(t))_{t\in[0,T]}$ such that for all $t\in(0,T]$ the following
conditions hold:
\begin{subequations} \label{eq:MS}
\begin{alignat}{2}
-\Delta u &=0 \quad &&\text{in } \Omega \setminus \Gamma(t),\label{eq:MSa}\\
u &= \varkappa \quad &&\text{on } \Gamma(t),\label{eq:MSb}\\
\left[\frac{\partial u}{\partial{\vec\nu}}\right]_{\Gamma(t)}
 &= -\mathcal{V}\quad && \text{on } \Gamma(t),\label{eq:MSc}\\
\frac{\partial u}{\partial {\vec\nu_\Omega}} &= 0 \quad && \text{on } 
\partial\Omega,\label{eq:MSd}
\end{alignat}
\end{subequations}
where $\vec\nu$ is the outer unit normal of $\Gamma(t)$, 
$\varkappa$ is its mean curvature,
$[\cdot]_{\Gamma(t)}$ denotes the jump of a quantity across the interface
$\Gamma(t)$ and $\mathcal{V}$ is the normal velocity of 
$(\Gamma(t))_{t\in[0,T]}$. Here our sign convention is such that the unit
sphere has mean curvature $\varkappa = 1 - d < 0$.

The problem \eqref{eq:MS} is usually called the Mullins--Sekerka problem,
or the two-sided Mullins--Sekerka flow,
and geometrically it can be viewed as a prototype for a 
curvature driven interface evolution that involves
quantities defined in the bulk regions surrounding the interface.
Alternative names for \eqref{eq:MS} in the literature are Hele--Shaw flow 
with surface tension, or quasi-static Stefan problem. For theoretical
results on the existence of strong and weak solutions to \eqref{eq:MS} 
we refer to \cite{Chen93,ChenHY96,EscherS97} and the references therein.
Physically, the system \eqref{eq:MS} was derived as a model 
for solidification and liquidation of materials of negligible specific heat,
\cite{MullinsS63}. In addition, the Mullins--Sekerka problem arises as the
sharp interface limit of the non-degenerate Cahn--Hilliard equation,
as was proved in \cite{AlikakosBC94}. Here we recall that the Cahn--Hilliard
equation models the process of phase separation and coarsening in melted 
alloys, \cite{CahnH58}. 

As regards the numerical approximation of the Mullins--Sekerka problem
\eqref{eq:MS}, several different approaches are available from the literature.
Approximations based on a boundary integral formulation can be found in e.g.\
\cite{BatesCD95,ZhuCH96,Mayer00}, while a front-tracking method based on
parametric finite elements has been proposed in \cite{dendritic}. For a
finite difference approximation of a
levelset formulation we refer to \cite{ChenMOS97}, while finite element
approximations of phasefield models have been considered in
\cite{FengP04,FengP04a,eck,vch}.
In this paper we will consider a front-tracking method, where the numerical
approximation of the interface $\Gamma(t)$ is completely independent of the 
finite element mesh for the bulk equation \eqref{eq:MSa}. 
In fact, we will propose an improvement for the unfitted finite element
approximation that was introduced by the author together with John W.\ Barrett
and Harald Garcke in \cite{dendritic}. Here we will put particular emphasis on
the conservation of physically relevant properties on the discrete level.

By way of motivation, we observe that it is not difficult to prove that a
solution to the Mullins--Sekerka problem \eqref{eq:MS} reduces the surface 
area $|\Gamma(t)|$, while it maintains the volume of the enclosed domain
$\Omega_-(t)$. In particular, it holds that
\begin{equation} \label{eq:dte}
\ddt |\Gamma(t)| 
= \ddt \int_{\Gamma(t)} 1 \dH{d-1}
= - \int_{\Gamma(t)} \varkappa\mathcal{V} \dH{d-1}
= -\int_{\Omega} |\nabla u|^2 \dL{d} \leq 0
\end{equation}
and
\begin{equation} \label{eq:dtv}
\ddt \vol(\Omega_-(t)) = \int_{\Gamma(t)} \mathcal{V} \dH{d-1}
= 0,
\end{equation}
see e.g.\ Remark~105 in \cite{bgnreview}. We remark that these properties
motivate the interpretation of \eqref{eq:MS} as a volume
preserving gradient flow of the surface area. 
Examples for volume preserving gradient flows of the surface area that only
depend on geometric properties of the interface are the conserved mean
curvature flow and surface diffusion, \cite{CahnT94,TaylorC94}. 
In contrast, the flow \eqref{eq:MS} also depends on the field $u$ that is
defined in the bulk.
A detailed description of the gradient flow structure for \eqref{eq:MS} 
can be found in \cite[Appendix~A]{dendritic}. 
Clearly, for a numerical approximation of \eqref{eq:MS} it would be highly
desirable to have a discrete analogue of the energy dissipation law 
\eqref{eq:dte} and the volume conservation property \eqref{eq:dtv}. 
The fully discrete method from \cite{dendritic} satisfies a discrete analogue
of \eqref{eq:dte}, in particular it is unconditionally stable. But a discrete
version of \eqref{eq:dtv} does not hold. That means that for large time steps,
and in certain situations, a significant loss of mass can be observed in
computations. On utilizing very recent ideas from \cite{JiangL21,BaoZ21},
we will appropriately adapt the fully discrete scheme from \cite{dendritic} 
to obtain a new method for \eqref{eq:MS} that satisfies discrete analogues of 
both \eqref{eq:dte} and \eqref{eq:dtv}. We believe this is the first such
fully discrete approximation of \eqref{eq:MS} in the literature.

In many physical applications, e.g.\ when considering the solidification
or liquidation of materials, the density of the interfacial energy is
directionally dependent. A typical example for such an anisotropic surface
energy is
\begin{equation} \label{eq:Egamma}
|\Gamma(t)|_\gamma = \int_{\Gamma(t)} \gamma(\vec\nu) \dH{d-1},
\end{equation}
where $\gamma$ is a given anisotropy function.
We refer to \cite{TaylorCH92,BellettiniNP99,DeckelnickDE05,Giga06} 
for more details on
anisotropic surface energies. On defining the anisotropic mean curvature
$\varkappa_\gamma$ as the first variation of \eqref{eq:Egamma}, so that e.g.\
$\ddt |\Gamma(t)|_\gamma 
= - \int_{\Gamma(t)} \varkappa_\gamma\mathcal{V} \dH{d-1}$, we can introduce
the anisotropic Mullins--Sekerka problem by replacing $\varkappa$ with
$\varkappa_\gamma$ in \eqref{eq:MS}. Then the energy dissipation \eqref{eq:dte} 
and volume conservation \eqref{eq:dtv} hold as before, where of course in the
former we need to replace $|\Gamma(t)|$ with $|\Gamma(t)|_\gamma$ and 
$\varkappa$ with $\varkappa_\gamma$. The numerical method we discuss in this
paper, by virtue of being derived from the scheme in \cite{dendritic},
can deal with the anisotropic Mullins--Sekerka problem as well. In addition,
for a class of anisotropies that was first proposed in \cite{triplejANI,ani3d},
the anisotropic scheme will still be structure preserving, in the sense that
discrete analogues of the anisotropic \eqref{eq:dte} and \eqref{eq:dtv} 
will hold.

In summary, the novel fully practical and fully discrete numerical method 
proposed in this paper has the following properties:
\begin{itemize}[itemsep=0.5pt,topsep=-3pt]
\item 
The method is unconditionally stable, i.e.\ it mimics \eqref{eq:dte} on the
discrete level.

\item
The volume of the two phases, i.e.\ the interior and the exterior of the
interface, is conserved exactly, as a fully discrete analogue to 
\eqref{eq:dtv}. 

\item
The polyhedral interface approximation maintains a nice mesh property,
leading to asymptotically equidistributed polygonal curves in the case $d=2$
for an isotropic surface energy.

\item
The method is unfitted, meaning mesh deformations of the bulk mesh are 
avoided, and no remeshings of the bulk triangulation are necessary.

\item
The method can take an anisotropic surface energy into account, meaning that
a discrete analogue of the anisotropic generalization of \eqref{eq:dte} 
still holds on the fully discrete level.
\end{itemize}

The remainder of the paper is organized as follows. In Section~\ref{sec:weak}
we introduce a weak formulation for the Mullins--Sekerka problem \eqref{eq:MS}
on which our finite element method is going be based. We also state a
semidiscrete continuous-in-time approximation and briefly analyse its
properties. Our novel fully discrete finite element approximation is presented
and analysed in Section~\ref{sec:fd}, where in order to focus on the structure 
preserving aspect of the method, we at first concentrate on the isotropic case.
Subsequently, in Section~\ref{sec:ani}, we discuss
the extension of the weak formulation and the finite element scheme to the
anisotropic case. Finally, in Section~\ref{sec:nr} we consider several
numerical simulations for the introduced numerical method, including some
convergence experiments.

\setcounter{equation}{0}
\section{Weak formulation and semidiscrete approximation} \label{sec:weak}
Our parametric finite element method will be based on a suitable weak
formulation of \eqref{eq:MS}, which we introduce in this section. Here we
follow the notation and presentation from the recent review article
\cite{bgnreview}.

Let 
\[\GT = \bigcup_{t\in[0,T]} (\Gamma(t)\times\{t\})\] 
be a smooth evolving hypersurface, such that for
every $t \in [0,T]$ the closed hypersurface $\Gamma(t) \subset \Omega$
partitions the domain $\Omega$ 
into two phases: the interior $\Omega_-(t)$ and the
exterior $\Omega_+(t) = \Omega \setminus \overline{\Omega_-(t)}$, so that
$\partial\Omega_-(t) = \Gamma(t)$. In what follows, we will often not
distinguish between $\Gamma(t) \times \{t\}$ and $\Gamma(t)$. Moreover, as we
are interested in a parametric formulation of the evolving interface, we assume
that $\vec x : \Upsilon\times [0,T] \to \bR^d$ is a global parameterization
of $(\Gamma(t))_{t\in[0,T]}$, where $\Upsilon \subset \bR^d$ is a smooth
reference manifold. We recall that the induced full velocity of
$\Gamma(t)$ is defined by
\begin{equation*}
\vec{\mathcal{V}}(\vec x(\vec q,t),t) = (\partial_t\vec x) (\vec q,t) 
\quad\forall\ (\vec q,t) \in \Upsilon \times [0,T],
\end{equation*}
and satisfies $\vec{\mathcal{V}} \cdot \vec\nu = \mathcal{V}$.
Multiplying \eqref{eq:MSa} with a test
function $\phi \in H^1(\Omega)$, integrating over $\Omega$ 
and performing integration by parts yields
\[
0 = \int_{\Omega_-(t)\cup\Omega_+(t)} \phi \Delta u \dL{d}
= \int_{\partial\Omega} \phi\frac{\partial u}{\partial \vec\nu_\Omega} \dH{d-1}
-\int_{\Gamma(t)}\phi 
\left[\frac{\partial u}{\partial \vec\nu}\right]_{\Gamma(t)}\dH{d-1}
-\int_\Omega \nabla u \cdot \nabla\phi \dL{d} ,
\]
which in view of the conditions \eqref{eq:MSc} and \eqref{eq:MSd} reduces to
\begin{equation*} 
0 = \int_{\Gamma(t)}\phi \mathcal{V}\dH{d-1}
-\int_\Omega \nabla u \cdot \nabla\phi \dL{d} .
\end{equation*}
The only other ingredient needed for the weak formulation is the well-known
variational formulation of mean curvature, given by
\begin{equation} \label{eq:varkappa}
\int_{\Gamma(t)} \varkappa\vec\eta\cdot\vec\nu + \nabs\vec\id: \nabs\vec\eta
 \dH{d-1} 
= 0 \qquad \vec\eta \in [H^1(\Gamma(t))]^d,
\end{equation}
where $\vec\id$ denotes the identity function in $\bR^d$ and $\nabs$ is
the surface gradient on $\Gamma(t)$,
see e.g.\ Remark~22 in \cite{bgnreview}. Hence, on denoting the
$L^2$--inner products over $\Omega$ and $\Gamma(t)$ by 
$(\cdot,\cdot)$ and $\langle\cdot,\cdot\rangle_{\Gamma(t)}$, respectively,
we can state the weak formulation as follows.

Given a closed hypersurface $\Gamma(0) \subset \Omega$, 
we seek an evolving hypersurface $(\Gamma(t))_{t\in[0,T]}$
that separates $\Omega$ into $\Omega_-(t)$ and $\Omega_+(t)$,
with a global parameterization and induced velocity field $\vec{\mathcal{V}}$,
and $\varkappa : L^2(\GT)$ as well as 
$u : \Omega \times [0,T] \to \bR$, such that for almost all $t \in (0,T)$
it holds for $(u(\cdot,t),\vec {\mathcal V}(\cdot,t),\varkappa(\cdot,t))\in 
H^1(\Omega) \times [L^2(\Gamma(t))]^d \times L^2(\Gamma(t))$ that 
\begin{subequations} \label{eq:WMS}
\begin{align} \label{eq:WMS1}
\left(\nabla u,\nabla\phi\right) -
\left\langle \vec{\mathcal{V}},\phi\vec\nu \right\rangle_{\Gamma(t)} & = 0
\qquad\forall\ \phi\in H^1(\Omega),\\
\left\langle u - \varkappa,\chi \right\rangle_{\Gamma(t)} 
& = 0 \qquad\forall\ \chi\in L^2(\Gamma(t))\label{eq:WMS2},\\
\left\langle \varkappa\vec\nu,\vec\eta \right\rangle_{\Gamma(t)} 
+ \left\langle \nabs\vec\id , \nabs\vec\eta 
\right\rangle_{\Gamma(t)}& =0 \qquad
\forall\ \vec\eta\in [H^1(\Gamma(t))]^d. \label{eq:WMS3}
\end{align}
\end{subequations}
Clearly, choosing $\phi = u$ in \eqref{eq:WMS1} and 
$\chi = \mathcal{V} = \vec{\mathcal{V}} \cdot \vec\nu$ in \eqref{eq:WMS2} 
yields the energy dissipation law
\eqref{eq:dte}, while choosing $\phi = 1$ in \eqref{eq:WMS1} leads to the
volume conservation property \eqref{eq:dtv}. 
Mimicking these testing procedures on the discrete level will be crucial to 
prove the structure preserving aspect of our finite element approximations.

For the numerical approximation of \eqref{eq:WMS} we first introduce the
necessary finite element space in the bulk. To this end, we assume that
$\Omega$ is a polyhedral domain. Then let $\mathcal{T}^h$ be a
regular partitioning of $\Omega$ into disjoint open simplices, so that
$\overline{\Omega}=\cup_{\sigmaO\in\mathcal{T}^h}\overline{\sigmaO}$, 
see \cite{Ciarlet78}.
Associated with $\mathcal{T}^h$ is the finite element space 
\begin{equation} \label{eq:Sh}
S^h = \left\{\chi \in C^0(\overline{\Omega}) : \chi_{\mid_{\sigmaO}}
\text{ is affine } \forall\ \sigmaO \in \mathcal{T}^h\right\} 
\subset H^1(\Omega). 
\end{equation}
In addition we need appropriate parametric finite element spaces.
Let a polyhedral hypersurface $\Gamma^h \subset \bR^d$ be given by
\begin{equation} \label{eq:Gammah}
\Gamma^h = \bigcup_{j=1}^J \overline{\sigma_j},
\end{equation}
where $\{\sigma_j\}_{j=1}^J$ is a family of disjoint, 
(relatively) open $(d-1)$-simplices, such that
$\overline{\sigma_i}\cap\overline{\sigma_j}$ for $i\not=j$ is either empty or
a common $k$-simplex of $\overline{\sigma_i}$ and $\overline{\sigma_j}$, 
$0 \leq k < d$. We denote the vertices of $\Gamma^h$
by $\{\vec q_k\}_{k=1}^K$, and assume that the vertices of $\sigma_j$
are given by $\{\vec q_{j,k}\}_{k=1}^{d}$, $j=1,\ldots,J$.
Here the numbering of the local vertices is assumed to be such that
\begin{equation} \label{eq:nuh}
 \vec\nu^h =
\frac{(\vec q_{j,2}-\vec q_{j,1}) \land \cdots \land
(\vec q_{j,d}-\vec q_{j,1})}{|(\vec q_{j,2}-\vec q_{j,1}) \land
\cdots \land (\vec q_{j,d}-\vec q_{j,1})|}
\qquad\text{on } \sigma_j, \quad j = 1,\ldots,J ,
\end{equation}
defines the outer normal $\vec\nu^h \in [L^\infty(\Gamma^h)]^d$ to the interior
$\Omega^h_-$ of $\Gamma^h = \partial\Omega^h_-$.
Here we recall the definition of the wedge product from 
\cite[Definition~45]{bgnreview}, i.e.\ for 
$\vec v_1, \ldots, \vec v_{d-1} \in \bR^d$, the wedge product 
is the unique vector such that 
$\vec b \cdot (\vec v_1\land\cdots\land\vec v_{d-1})
= \det(\vec v_1,\ldots,\vec v_{d-1}, \vec b)$ for all $\vec b\in\bR^d$. 
It follows that it is the usual cross product
of two vectors in $\bR^3$, and the anti-clockwise rotation 
through $\frac\pi2$ of a vector in $\bR^2$. We note also that
\begin{equation} \label{eq:detsigma}
|\sigma_j| = 
\frac{1}{d-1}|(\vec q_{j,2}-\vec q_{j,1}) \land \cdots \land
(\vec q_{j,d}-\vec q_{j,1})|.
\end{equation}

We define the finite element spaces of continuous
piecewise linear functions on $\Gamma^h$ via
\[
\Whh = \{\chi \in C^0(\Gamma^h) : \chi_{\mid_{\sigma_j}}
\text{ is affine for $j=1, \ldots, J$} \}\quad\text{and}\quad \Vhh = [\Whh]^d.
\]
We let $\{\phi^{\Gamma^h}_k\}_{k=1}^K$ denote
the standard basis of $\Whh$, i.e.\ 
\[
\phi^{\Gamma^h}_i(\vec q_j) = \delta_{ij}, \qquad i,j = 1,\ldots,K.
\]
Moreover, we let $\pi_{\Gamma^h}:C^0(\Gamma^h) \to \Whh$ be the standard 
interpolation operator, and let 
$\left\langle \cdot, \cdot \right\rangle_{\Gamma^h}$
denote the $L^2$--inner product on $\Gamma^h$.
For two piecewise continuous functions
$u,v\in L^\infty(\Gamma^h)$, with possible jumps
across the edges of $\{\sigma_j\}_{j=1}^J$,
we introduce the mass lumped inner product
$\left\langle\cdot,\cdot\right\rangle^h_{\Gamma^h}$ as
\begin{equation}
\left\langle u, v \right\rangle^h_{\Gamma^h} =
\frac1{d}\sum_{j=1}^J |\sigma_j|
\sum_{k=1}^{d} (uv)((\vec q_{j,k})^-),
\label{eq:ip0}
\end{equation}
where $u((\vec q)^-) =
\underset{\sigma_j\ni \vec p\to \vec q}{\lim} u(\vec p)$.
The definition \eqref{eq:ip0}
is naturally extended to vector- and tensor-valued functions.
On recalling \eqref{eq:nuh}, we define the vertex normal vector 
$\vec\omega^h \in \Vhh$ to be the mass-lumped $L^2$--projection 
of $\vec\nu^h$ onto $\Vhh$, i.e.\
\begin{equation} \label{eq:omegah}
\left\langle \vec\omega^h, \vec\varphi \right\rangle_{\Gamma^h}^h 
= \left\langle \vec\nu^h, \vec\varphi \right\rangle_{\Gamma^h}
\quad\forall\ \vec\varphi\in\Vhh.
\end{equation}

{From} now on, we let \[\GhT = \bigcup_{t\in[0,T]} (\Gamma^h(t)\times\{t\})\]
be an evolving polyhedral hypersurface, so that  
$\Gamma^h(t)$, for each $t\in[0,T]$, is a polyhedral surface of the form 
\eqref{eq:Gammah} for fixed $J$ and $K$.
That is, $\Gamma^h(t)$ is defined through its elements 
$\{\sigma^h_j(t)\}_{j=1}^J$ and its vertices $\{\vec q^h_k(t)\}_{k=1}^K$.
We will often not distinguish between $\GhT$ and $(\Gamma^h(t))_{t\in[0,T]}$.
Then the full velocity of $\Gamma^h(t)$ is defined by
\begin{equation} \label{eq:Vh}
\vec{\mathcal{V}}^h(\vec z, t) = \sum_{k=1}^{K}
\left[\ddt\vec q_k(t)\right] \phi^{\Gamma^h(t)}_k(\vec z) 
\quad\forall\ (\vec z,t) \in \GhT.
\end{equation}
We also define the finite element spaces
\[
\WhGhT = \{ \chi \in C^0(\GhT) : 
\chi(\cdot, t) \in \Wht \quad\forall\ t \in [0,T] \}
\quad\text{and}\quad \VhGhT = [\WhGhT]^d.
\]

Our unfitted semidiscrete finite element approximation of \eqref{eq:WMS} can
then be formulated as follows.
Given the closed polyhedral hypersurface $\Gamma^h(0)$,
find an evolving polyhedral hypersurface $(\Gamma^h(t))_{t\in[0,T]}$, 
that separates $\Omega$ into $\Omega^h_-(t)$ and $\Omega^h_+(t)$,
with induced velocity $\vec{\mathcal{V}}^h \in \VhGhT$,
and $\kappa^h \in \WhGhT$ as well as $u^h \in S^h\times(0,T]$, 
such that for all $t \in (0,T]$ it holds for
$(u^h(\cdot,t), \vec{\mathcal{V}}^h(\cdot, t), \kappa^h(\cdot,t))\in S^h \times
\Vht \times \Wht$ that
\begin{subequations} \label{eq:sdMS}
\begin{align} \label{eq:sdMS1}
\left(\nabla u^h,\nabla\phi\right) -
\left\langle \pi_{\Gamma^h(t)}
\left[\vec{\mathcal{V}}^h \cdot \vec\omega^h\right],
\phi\right\rangle_{\Gamma^h(t)}^{(h)}  & = 0
\qquad\forall\ \phi\in S^h,\\
\left\langle u^h,\chi \right\rangle_{\Gamma^h(t)}^{(h)} 
- \left\langle\kappa^h,\chi \right\rangle_{\Gamma^h(t)}^h
& = 0 \qquad\forall\ \chi\in \Wht,\label{eq:sdMS2}\\
\left\langle \kappa^h\vec\omega^h,\vec\eta \right\rangle_{\Gamma^h(t)}^h 
+ \left\langle \nabs\vec\id , \nabs\vec\eta 
\right\rangle_{\Gamma^h(t)}& =0 \qquad
\forall\ \vec\eta\in \Vht, \label{eq:sdMS3}
\end{align}
\end{subequations}
where the surface gradients $\nabs$ in \eqref{eq:sdMS3} are defined piecewise
on the polyhedral surface $\Gamma^h(t)$. 
Here and throughout, the notation $\cdot^{(h)}$ means an expression with or
without the superscript $h$. That is, the scheme \eqref{eq:sdMS}$^h$ employs
numerical integration in the two relevant terms in \eqref{eq:sdMS1} and
\eqref{eq:sdMS2}, while the scheme \eqref{eq:sdMS} uses true integration in
these two terms. 
We also remark that thanks to \eqref{eq:omegah} and the piecewise constant 
nature of $\vec\nu^h$, the first term in \eqref{eq:sdMS3} is equivalent
to $\langle \kappa^h\vec\nu^h,\vec\eta \rangle_{\Gamma^h(t)}^h$. We prefer
to write it in terms of $\vec\omega^h$ to make the testing procedure in the
analysis easier to follow. Before we present a proof for the structure
preserving properties of \eqref{eq:sdMS}$^{(h)}$, 
we recall the following fundamental results from \cite{bgnreview}.

\begin{lemma}
Let $(\Gamma^h(t))_{t\in[0,T]}$ be an evolving polyhedral hypersurface. Then it
holds that
\begin{equation} \label{eq:dteh}
\ddt\left|\Gamma^h(t)\right| 
= \left\langle \nabs\vec\id, \nabs\vec{\mathcal{V}}^h 
\right\rangle_{\Gamma^h(t)} 
\end{equation}
and
\begin{equation} \label{eq:dtvh}
\ddt\vol(\Omega^h_-(t)) 
= \left\langle \vec{\mathcal{V}}^h,\vec\nu^h\right\rangle_{\Gamma^h(t)}.
\end{equation}
\end{lemma}
\begin{proof}
The result \eqref{eq:dteh} directly follows from Theorem~70 and Lemma~9 in
\cite{bgnreview}, while a proof for \eqref{eq:dtvh} is given in
Theorem~71 in \cite{bgnreview}. 
\end{proof}

We are now in a position to prove energy decay, volume conservation and 
good mesh quality properties for a solution of \eqref{eq:sdMS}$^{(h)}$.
Here for the definition of a conformal polyhedral surface we recall
Definition~60 from \cite{bgnreview}:
\begin{definition} \label{def:conformal}
A closed polyhedral hypersurface $\Gamma^h$,
with unit normal $\vec\nu^h$, is called a conformal polyhedral
hypersurface, if there exists a $\kappa^h \in \Whh$ such that
\begin{equation} \label{eq:conformal}
\left\langle \kappa^h\,\vec\nu^h, \vec\eta\right\rangle^h_{\Gamma^h}
= - \left\langle \nabs\,\vec\id, \nabs\,\vec\eta \right\rangle_{\Gamma^h}
\quad\forall\ \vec\eta \in \Vhh\,.
\end{equation}
\end{definition}
The discussion in \cite[\S4.1]{gflows3d} indicates that for $d=3$ surfaces
satisfying Definition~\ref{def:conformal}
are characterized by a good mesh quality, and this is confirmed by
a large body of numerical evidence in e.g.\ 
\cite{gflows3d,spurious,fluidfbp,tpfs}. 
On the other hand, for $d=2$ it is shown in \cite[Theorem~62]{bgnreview}
that any conformal polygonal curve is weakly equidistributed.

\begin{theorem} \label{thm:sd}
Let $(u^h, \GhT, \kappa^h)$ be a solution of \eqref{eq:sdMS}$^{(h)}$. 
Then it holds that
\begin{equation} \label{eq:MSdteh}
\ddt\left|\Gamma^h(t)\right| + \left(\nabla u^h, \nabla u^h\right) = 0.
\end{equation}
Moreover we have that
\begin{equation} \label{eq:MSdtvh}
\ddt \vol(\Omega^h_-(t)) = 0.
\end{equation}
Finally, for any $t \in (0,T]$, it holds that
$\Gamma^h(t)$ is a conformal polyhedral surface.
In particular, for $d=2$, any two neighbouring elements of the curve 
$\Gamma^h(t)$ either have equal length, or they are parallel.
\end{theorem}
\begin{proof}
Choosing $\phi = u^h(\cdot,t) \in S^h$ in \eqref{eq:sdMS1}, 
$\chi = \pi_{\Gamma^h(t)}[\vec{\mathcal{V}}^h \cdot \vec\omega^h] 
\in \Wht$ in \eqref{eq:sdMS2} and
$\vec\eta = \vec{\mathcal{V}}^h(\cdot, t) \in \Vht$ in \eqref{eq:sdMS3} 
gives, on recalling \eqref{eq:dteh}, that
\begin{align*}
\ddt\left|\Gamma^h(t)\right| & 
= \left\langle \nabs\vec\id, \nabs\vec{\mathcal{V}}^h 
\right\rangle_{\Gamma^h(t)} 
= - \left\langle \kappa^h\vec\omega^h,\vec{\mathcal{V}}^h 
\right\rangle_{\Gamma^h(t)}^h 
= - \left\langle \kappa^h , \pi_{\Gamma^h(t)}
\left[\vec{\mathcal{V}}^h \cdot \vec\omega^h\right] 
\right\rangle_{\Gamma^h(t)}^h \\ &
= - \left\langle \pi_{\Gamma^h(t)}
\left[\vec{\mathcal{V}}^h \cdot \vec\omega^h\right],
u^h\right\rangle_{\Gamma^h(t)}^{(h)} 
= - \left(\nabla u^h,\nabla u^h\right) ,
\end{align*}
which implies \eqref{eq:MSdteh}. 
Moreover, choosing $\phi=1$ in \eqref{eq:sdMS1} 
and noting \eqref{eq:omegah}, on recalling \eqref{eq:dtvh}, yields that 
\begin{equation*}
\ddt\vol(\Omega^h_-(t)) 
= \left\langle \vec{\mathcal{V}}^h,\vec\nu^h\right\rangle_{\Gamma^h(t)} 
= \left\langle \vec{\mathcal{V}}^h,\vec\omega^h\right\rangle_{\Gamma^h(t)}^h 
= \left\langle \pi_{\Gamma^h(t)}
\left[\vec{\mathcal{V}}^h \cdot \vec\omega^h\right],
1\right\rangle_{\Gamma^h(t)}^{(h)} 
= \left(\nabla u, \nabla 1\right)
= 0,
\end{equation*}
which is \eqref{eq:MSdtvh}.
Finally, the mesh properties for $\Gamma^h(t)$ follow directly from the side
condition \eqref{eq:sdMS3}, thanks to 
Definition~\ref{def:conformal}
and Theorem~62 in \cite{bgnreview},
on noting that 
$\langle \kappa^h\vec\omega^h,\vec\eta \rangle_{\Gamma^h(t)}^h=
\langle \kappa^h\vec\nu^h,\vec\eta \rangle_{\Gamma^h(t)}^h$.
\end{proof}

The motivation for the choices of numerical quadrature in \eqref{eq:sdMS} is
apparent now. We employ mass lumping for the first term in \eqref{eq:sdMS3} to
ensure the good mesh properties. This in turn enforces the use of mass lumping
in the second term in \eqref{eq:sdMS2}, in order to guarantee stability. 
Finally, for the two bulk-interface integrals we allow a choice between true 
integration and mass lumping, the latter being considerably easier to
implement; see the beginning of Section~\ref{sec:nr} below.

\setcounter{equation}{0}
\section{Fully discrete approximation} \label{sec:fd}
The aim of this section is to introduce a fully practical fully discrete
approximation of \eqref{eq:sdMS}$^{(h)}$ that maintains the structure preserving
properties from Theorem~\ref{thm:sd}.

Let $0 = t_0 < t_1<\ldots<t_M=T$ form a partition of the time interval $[0,T]$
with time steps $\ttau_m= t_{m+1}-t_m$, $m=0,\ldots,M-1$.
The main idea going back to the seminal paper \cite{Dziuk91} is now to 
construct polyhedral hypersurfaces $\Gamma^m$, which approximate the
true continuous solutions $\Gamma(t_m)$, in such a way that for $m\geq 0$
we obtain $\Gamma^{m+1} = \vec X^{m+1}(\Gamma^m)$ for a 
parameterization $\vec X^{m+1} \in \Vhm$. In addition we consider a sequence of
bulk triangulations $\mathcal{T}^m$ with associated finite element spaces
$S^m$, $m=0,\ldots,M-1$, similarly to \eqref{eq:Sh}. 
For motivational purposes, we first recall the linear fully discrete 
approximation of \eqref{eq:sdMS} from \cite{bgnreview}. 

Let the closed polyhedral hypersurface $\Gamma^0$ be an approximation of
$\Gamma(0)$. 
Then, for $m=0,\ldots,M-1$, find 
$(U^{m+1},\vec X^{m+1},\kappa^{m+1}) \in S^m \times \Vhm \times \Whm$ such that
\begin{subequations} \label{eq:DMS}
\begin{align}\label{eq:DMS1}
\left(\nabla U^{m+1}, \nabla\varphi\right) 
- \left\langle \pi_{\Gamma^m}\left[
\frac{\vec X^{m+1}-\vec\id}{\ttau_m} \cdot\vec\omega^m \right], \varphi
\right\rangle_{\Gamma^m}^{(h)} & = 0 \qquad\forall\ \varphi \in S^m,\\
\left\langle U^{m+1},\chi\right\rangle_{\Gamma^m}^{(h)} 
-\left\langle\kappa^{m+1},\chi\right\rangle^h_{\Gamma^m} 
&= 0 \qquad\forall\ \chi \in \Whm ,\label{eq:DMS2}\\
\left\langle \kappa^{m+1}\vec\omega^m,\vec\eta\right\rangle^h_{\Gamma^m} +
  \left\langle\nabs\vec X^{m+1},\nabs\vec\eta\right\rangle_{\Gamma^m}
&= 0
\qquad\forall\ \vec\eta \in \Vhm \label{eq:DMS3}
\end{align}
\end{subequations}
and set $\Gamma^{m+1} = \vec X^{m+1}(\Gamma^m)$. We observe that 
\eqref{eq:DMS} corresponds to \cite[(119)]{bgnreview}, which was first
introduced in \cite[(3.5)]{dendritic}. Under mild conditions on $\Gamma^m$,
existence and uniqueness for the linear system \eqref{eq:DMS}$^{(h)}$ 
can be shown.
Moreover, solutions to \eqref{eq:DMS}$^{(h)}$ are unconditionally stable, see
\cite[Theorem~109]{bgnreview}. However, in general the volume of the interiors
$\Omega^{m+1}_-$ and $\Omega^{m}_-$ of $\Gamma^{m+1}$ and $\Gamma^m$,
respectively, will
differ, meaning that the fully discrete scheme \eqref{eq:DMS}$^{(h)}$ 
is not volume preserving. 
The reason for this behaviour is the explicit approximation
of $\vec\omega^h$ from \eqref{eq:sdMS1} in \eqref{eq:DMS1}. Following the
recent ideas in \cite{BaoZ21}, we now investigate a semi-implicit approximation
of $\vec\omega^h$ which will lead to a volume preserving approximation.

Given a sequence of polyhedral surfaces $(\Gamma^m)_{m=0}^M$, where each
$\Gamma^m$ is defined through its
vertices $\{\vec q^m_k\}_{k=1}^K$ and elements 
$\{\sigma^m_j\}_{j=1}^J$, we define the piecewise-linear-in-time family of
polyhedral surfaces $(\hat\Gamma^h(t))_{t \in[0,T]}$ via
\begin{equation*} 
\hat\Gamma^h(t) = \frac{t_{m+1}-t}{\ttau_m} \Gamma^m +
\frac{t - t_m}{\ttau_m} \Gamma^{m+1}, \qquad t \in [t_m, t_{m+1}],
\end{equation*}
which means that the polyhedral surface $\hat\Gamma^h(t)$ 
is induced by the vertices
\begin{equation*} 
\hat q^h_k(t) = \frac{t_{m+1}-t}{\ttau_m} \vec q^m_k +
\frac{t - t_m}{\ttau_m} \vec q^{m+1}_k, \qquad t \in [t_m, t_{m+1}],
\end{equation*}
for $k=1,\ldots,K$. We note that $\hat\Gamma^h(t_m) = \Gamma^m$,
$m=0,\ldots,M$.
Then it immediately follows from \eqref{eq:Vh} that
\begin{equation*} 
\vec{\mathcal{V}}^h(\cdot, t) = \frac1{\ttau_m} \sum_{k=1}^{K}
\left[\vec q_k^{m+1} - \vec q_k^m \right]\phi^{\hat\Gamma^h(t)}_k
\quad\text{on } \hat\Gamma^h(t),\qquad t \in (t_m, t_{m+1}).
\end{equation*}
On denoting the interior of $\hat\Gamma^h(t)$ by $\hat\Omega_-^h(t)$,
with outer unit normal $\hat\nu^h(t)$,
the fundamental theorem of calculus, together with \eqref{eq:dtvh}, yields
that
\begin{align} \label{eq:volvol}
\vol(\Omega_-^{m+1}) - \vol(\Omega_-^{m}) & =
\vol(\hat\Omega_-^h(t_{m+1})) - \vol(\hat\Omega_-^h(t_{m})) =
\int_{t_m}^{t_{m+1}} \ddt \vol(\hat\Omega_-^h(t)) \dt \nonumber \\ &
=\int_{t_m}^{t_{m+1}} 
\left\langle \vec{\mathcal{V}}^h,\hat\nu^h\right\rangle_{\hat\Gamma^h(t)} \dt
\nonumber \\ &
=\int_{t_m}^{t_{m+1}} \left\langle
\frac1{\ttau_m} \sum_{k=1}^{K}
\left[\vec q_k^{m+1} - \vec q_k^m \right]\phi^{\hat\Gamma^h(t)}_k
 ,\hat\nu^h\right\rangle_{\hat\Gamma^h(t)}\dt
\nonumber \\ &
= \frac1{\ttau_m} \sum_{k=1}^{K}
\left[\vec q_k^{m+1} - \vec q_k^m \right] \cdot
\int_{t_m}^{t_{m+1}} \left( \int_{\hat\Gamma^h(t)}
\phi^{\hat\Gamma^h(t)}_k \hat\nu^h \dH{d-1} \right) \dt
\nonumber \\ &
= \frac1{\ttau_m} \sum_{k=1}^{K}
\left[\vec q_k^{m+1} - \vec q_k^m \right] \cdot
\int_{t_m}^{t_{m+1}} \left( \sum_{j=1}^J \int_{\hat\sigma^h_j(t)}
\phi^{\hat\Gamma^h(t)}_k \hat\nu^h \dH{d-1} \right) \dt 
\nonumber \\ &
= \frac1{\ttau_m} \sum_{k=1}^{K}
\left[\vec q_k^{m+1} - \vec q_k^m \right] \cdot
\int_{t_m}^{t_{m+1}} \left( \sum_{j=1}^J \int_{\sigma^m_j}
\phi^{\Gamma^m}_k \dH{d-1} \hat\nu^h\!\mid_{\hat\sigma^h_j(t)} 
\frac{|\hat\sigma^h_j(t)|}{|\sigma^m_j|} \right) \dt \nonumber \\ &
= \frac1{\ttau_m} \int_{t_m}^{t_{m+1}} \left( \sum_{j=1}^J \int_{\sigma^m_j}
\vec X^{m+1} - \vec\id \dH{d-1} \cdot
\hat\nu^h\!\mid_{\hat\sigma^h_j(t)} 
\frac{|\hat\sigma^h_j(t)|}{|\sigma^m_j|} \right) \dt\nonumber \\ &
= \sum_{j=1}^J \int_{\sigma^m_j}
\vec X^{m+1} - \vec\id \dH{d-1} \cdot
\frac1{\ttau_m |\sigma^m_j|} \int_{t_m}^{t_{m+1}} 
\hat\nu^h\!\mid_{\hat\sigma^h_j(t)} |\hat\sigma^h_j(t)| \dt,
\end{align}
where we have used the previously introduced notation 
$\vec X^{m+1} = \sum_{k=1}^K \phi^{\Gamma^m}_k \vec q^{m+1}_k \in \Vhm$.
The calculation in \eqref{eq:volvol} suggests the definition of the piecewise
constant vector $\vec\nu^{m+\frac12} \in [L^\infty(\Gamma^m)]^d$ by setting
\begin{equation} \label{eq:nuhalf}
 \vec\nu^{m+\frac12} =
\frac1{\ttau_m |\sigma^m_j|} \int_{t_m}^{t_{m+1}} 
\hat\nu^h\!\mid_{\hat\sigma^h_j(t)} |\hat\sigma^h_j(t)| \dt
\qquad\text{on } \sigma_j^m, \quad j = 1,\ldots,J .
\end{equation}
We note that $\vec\nu^{m+\frac12}$ can be interpreted as an averaged normal
vector for the linearly interpolated surfaces between $\Gamma^m$ and
$\Gamma^{m+1}$. Note also that in general $\vec\nu^{m+\frac12}$ will not have
unit length. Overall we have proven the following result, which generalizes
the corresponding results from Theorems~2.1 and 3.1 in
\cite{BaoZ21} to the case $d\geq2$.

\begin{lemma} \label{lem:volvol}
It holds that
\[
\vol(\Omega_-^{m+1}) - \vol(\Omega_-^{m}) 
= \left\langle \vec X^{m+1} - \vec\id , \vec\nu^{m+\frac12}
\right\rangle_{\Gamma^m}.
\]
\end{lemma}
\begin{proof}
The desired result follows immediately from \eqref{eq:volvol} and the
definition \eqref{eq:nuhalf}.  
\end{proof}

\begin{remark} \label{rem:nuhalf}
In practice, given $\Gamma^m$ and $\Gamma^{m+1}$, the vector
$\vec\nu^{m+\frac12}$ is remarkably easy to compute, since the integrand in
\eqref{eq:nuhalf} is a polynomial of degree $d-1$. In particular, it holds that
\[
\vec\nu^{m+\frac12}\!\mid_{\sigma^m_j} =
\frac1{\ttau_m} 
\int_{t_m}^{t_{m+1}} 
\frac{(\hat q^h_{j,2}(t)-\hat q^h_{j,1}(t)) \land \cdots \land
(\hat q^h_{j,d}(t)-\hat q^h_{j,1}(t))}{
|(\vec q^m_{j,2}-\vec q^m_{j,1}) \land
\cdots \land (\vec q^m_{j,d}-\vec q^m_{j,1})|}
 \dt ,
\]
where we have recalled \eqref{eq:nuh} and \eqref{eq:detsigma}. Using suitable
quadrature rules then yields in the case $d=2$ that
\begin{equation} \label{eq:nuhalf2}
\vec\nu^{m+\frac12}\!\mid_{\sigma^m_j}
= \frac12 \frac{(\vec q^m_{j,2}-\vec q^m_{j,1} 
+ \vec q^{m+1}_{j,2}-\vec q^{m+1}_{j,1})^\perp}{
|\vec q^m_{j,2}-\vec q^m_{j,1}|},
\end{equation}
where $\cdot^{\perp}$ denotes the anti-clockwise rotation 
through $\frac\pi2$ of a vector in $\bR^2$,
while for $d=3$ we obtain
\begin{align} \label{eq:nuhalf3}
\vec\nu^{m+\frac12}\!\mid_{\sigma^m_j}
& = \frac16 \frac{(\vec q^m_{j,2}-\vec q^m_{j,1}) \times
(\vec q^m_{j,3}-\vec q^m_{j,1})+(\vec q^{m+1}_{j,2}-\vec q^{m+1}_{j,1}) \times
(\vec q^{m+1}_{j,3}-\vec q^{m+1}_{j,1})}{
|(\vec q^m_{j,2}-\vec q^m_{j,1}) \times
 (\vec q^m_{j,3}-\vec q^m_{j,1})|} \nonumber \\ & \quad
+ \frac16 \frac{(\vec q^m_{j,2} - \vec q^m_{j,1} 
+ \vec q^{m+1}_{j,2}  - \vec q^{m+1}_{j,1}) \times
(\vec q^m_{j,3} - \vec q^m_{j,1} + \vec q^{m+1}_{j,3} - \vec q^{m+1}_{j,1})}{
|(\vec q^m_{j,2}-\vec q^m_{j,1}) \times
 (\vec q^m_{j,3}-\vec q^m_{j,1})|}
.
\end{align}
\end{remark}

Before we can apply the result from Lemma~\ref{lem:volvol} to the approximation
\eqref{eq:DMS}$^{(h)}$, 
we need to introduce a vertex based normal corresponding to
$\vec\nu^{m+\frac12}$. Analogously to \eqref{eq:omegah} we therefore define
$\vec\omega^{m+\frac12} \in \Vhm$ such that
\begin{equation} \label{eq:omegahalf}
\left\langle \vec\omega^{m+\frac12}, \vec\varphi \right\rangle_{\Gamma^m}^h 
= \left\langle \vec\nu^{m+\frac12}, \vec\varphi \right\rangle_{\Gamma^m}
\quad\forall\ \vec\varphi\in\Vhm.
\end{equation}

Now our novel fully discrete approximation of \eqref{eq:sdMS}$^{(h)}$ 
is given as follows. 

Let the closed polyhedral hypersurface $\Gamma^0$ be an approximation of
$\Gamma(0)$. 
Then, for $m=0,\ldots,M-1$, find 
$(U^{m+1},\vec X^{m+1},\kappa^{m+1}) \in S^m \times \Vhm \times \Whm$ 
and $\Gamma^{m+1} = \vec X^{m+1}(\Gamma^m)$ such that
\begin{subequations} \label{eq:fdMS}
\begin{align}\label{eq:fdMS1}
\left(\nabla U^{m+1}, \nabla\varphi\right) 
- \left\langle \pi_{\Gamma^m}\left[
\frac{\vec X^{m+1}-\vec\id}{\ttau_m} \cdot\vec\omega^{m+\frac12}\right],\varphi
\right\rangle_{\Gamma^m}^{(h)} & = 0 \qquad\forall\ \varphi \in S^m,\\
\left\langle U^{m+1},\chi\right\rangle_{\Gamma^m}^{(h)} 
-\left\langle\kappa^{m+1},\chi\right\rangle^h_{\Gamma^m} 
&= 0 \qquad\forall\ \chi \in \Whm ,\label{eq:fdMS2}\\
\left\langle \kappa^{m+1}\vec\omega^{m+\frac12},
\vec\eta\right\rangle^h_{\Gamma^m} +
\left\langle\nabs\vec X^{m+1},\nabs\vec\eta\right\rangle_{\Gamma^m}
&= 0
\qquad\forall\ \vec\eta \in \Vhm. \label{eq:fdMS3}
\end{align}
\end{subequations}
We note that in contrast to
\eqref{eq:DMS}$^{(h)}$, the scheme \eqref{eq:fdMS}$^{(h)}$ 
leads to a system of nonlinear
equations at each time level, because $\vec\omega^{m+\frac12}$ depends on 
$\vec X^{m+1}$.

The next theorem proves the structure preserving properties of the fully
discrete approximation \eqref{eq:fdMS}$^{(h)}$.

\begin{theorem} \label{thm:HDMS}
Let $(U^{m+1}, \vec X^{m+1},\kappa^{m+1})\in
S^m\times \Vhm\times \Whm$ be a solution to \eqref{eq:fdMS}$^{(h)}$. 
Then the enclosed volume is preserved, i.e.\
\begin{equation} \label{eq:thmvol}
\vol(\Omega_-^{m+1}) = \vol(\Omega_-^{m}) .
\end{equation}
In addition, if $d=2$ or $d=3$, then the solution satisfies
the stability estimate
\begin{equation} \label{eq:thmstab}
|\Gamma^{m+1}| + \ttau_m \left(\nabla U^{m+1}, \nabla U^{m+1} \right)
\leq |\Gamma^m|.
\end{equation}
\end{theorem}
\begin{proof} 
On choosing $\varphi=1$ in \eqref{eq:fdMS1}, it follows from
\eqref{eq:omegahalf} and Lemma~\ref{lem:volvol} that
\begin{align*}
0 & = \left\langle \pi_{\Gamma^m}\left[
\frac{\vec X^{m+1}-\vec\id}{\ttau_m} \cdot\vec\omega^{m+\frac12}\right],1
\right\rangle_{\Gamma^m}^{(h)}
= \left\langle \frac{\vec X^{m+1}-\vec\id}{\ttau_m}, \vec\omega^{m+\frac12}
\right\rangle_{\Gamma^m}^h
=\left\langle \frac{\vec X^{m+1}-\vec\id}{\ttau_m}, \vec\nu^{m+\frac12}
\right\rangle_{\Gamma^m} \\ &
= \frac1{\ttau_m} \left(\vol(\Omega_-^{m+1}) - \vol(\Omega_-^{m})\right) .
\end{align*}
This proves \eqref{eq:thmvol}. 
It remains to prove the stability bound. Here we choose
$\varphi= U^{m+1}$ in \eqref{eq:fdMS1}, 
$\chi = \pi_{\Gamma^m}[(\vec X^{m+1}-\vec\id)\cdot\vec\omega^{m+\frac12}]$ 
in \eqref{eq:fdMS2} and 
$\vec\eta = \vec X^{m+1}-\vec\id_{\mid_{\Gamma^m}}$ in \eqref{eq:fdMS3} 
in order to obtain
\begin{equation} \label{eq:step1}
\ttau_m\left(\nabla U^{m+1}, \nabla U^{m+1} \right) + 
\left\langle\nabs\vec X^{m+1},\nabs(\vec X^{m+1}-\vec\id) 
\right\rangle_{\Gamma^m} =0.
\end{equation}
Now we recall from Lemma~57 in \cite{bgnreview} the well-known bound
\begin{equation} \label{eq:stab2d3d}
\left\langle\nabs\vec X^{m+1},\nabs(\vec X^{m+1}-\vec\id) 
\right\rangle_{\Gamma^m} \geq |\Gamma^{m+1}| - |\Gamma^m|
\end{equation}
for the cases $d=2$ and $d=3$. Combining \eqref{eq:step1} and
\eqref{eq:stab2d3d} yields the desired result \eqref{eq:thmstab}. 
\end{proof}

In practice the system of nonlinear equations \eqref{eq:fdMS}$^{(h)}$ 
can be solved
with a simple lagged iteration. Given $\Gamma^m$, let $\Gamma^{m+1,0} =
\Gamma^m$. Then for $i \geq 0$ define $\vec\omega^{m+\frac12,i} \in \Vhm$
through \eqref{eq:omegahalf} and \eqref{eq:nuhalf}, but with $\Gamma^{m+1}$
replaced by $\Gamma^{m+1,i}$, and find
$(U^{m+1,i+1},\vec X^{m+1,i+1},\kappa^{m+1,i+1}) 
\in S^m \times \Vhm \times \Whm$ 
such that
\begin{subequations} \label{eq:itMS}
\begin{align}\label{eq:itMS1}
\left(\nabla U^{m+1,i+1}, \nabla\varphi\right) 
- \left\langle \pi_{\Gamma^m}\left[
\frac{\vec X^{m+1,i+1}-\vec\id}{\ttau_m} \cdot\vec\omega^{m+\frac12,i}\right],
\varphi \right\rangle_{\Gamma^m}^{(h)} & = 0 \qquad\forall\ \varphi \in S^m,\\
\left\langle U^{m+1,i+1},\chi\right\rangle_{\Gamma^m}^{(h)} 
-\left\langle\kappa^{m+1,i+1},\chi\right\rangle^h_{\Gamma^m} 
&= 0 \qquad\forall\ \chi \in \Whm ,\label{eq:itMS2}\\
\left\langle \kappa^{m+1,i+1}\vec\omega^{m+\frac12,i},
\vec\eta\right\rangle^h_{\Gamma^m} +
\left\langle\nabs\vec X^{m+1,i+1},\nabs\vec\eta\right\rangle_{\Gamma^m}
&= 0
\qquad\forall\ \vec\eta \in \Vhm \label{eq:itMS3}
\end{align}
\end{subequations}
and set $\Gamma^{m+1,i+1} = \vec X^{m+1,i+1}(\Gamma^m)$. The iteration can be
repeated until the stopping criterion
\begin{equation} \label{eq:stop}
\| \vec X^{m+1,i+1} - \vec X^{m+1,i} \|_\infty \leq \text{tol}
\end{equation}
is satisfied. Note that the existence of a unique solution to the linear 
system of equations \eqref{eq:itMS}$^{(h)}$, which is of the same form as 
\eqref{eq:DMS}$^{(h)}$, 
can be shown under mild assumptions on $\Gamma^m$, recall Theorem~109
in \cite{bgnreview}. 

\setcounter{equation}{0}
\section{Generalization to anisotropic surface energies} \label{sec:ani}

In this section we briefly discuss the extension of the finite element
approximation \eqref{eq:fdMS}$^{(h)}$ to the case of an anisotropic surface energy of
the form \eqref{eq:Egamma}, i.e.\
\[
|\Gamma(t)|_\gamma = \int_{\Gamma(t)} \gamma(\vec\nu) \dH{d-1}.
\]
On defining the anisotropic curvature through
\[
\varkappa_\gamma = - \nabs \cdot \gamma'(\vec\nu) \quad \text{on } \Gamma(t),
\]
where $\gamma'$ denotes the spatial gradient of $\gamma : \bR^d \to \bR$,
which itself is defined as a one-homogeneous extension of the originally given
density on the unit ball, we introduce the anisotropic analogue of 
\eqref{eq:MS} via
\begin{equation} \label{eq:aniMS}
-\Delta u =0 \quad\!\! \text{in } \Omega \setminus \Gamma(t),\quad
u = \varkappa_\gamma \quad\!\! \text{on } \Gamma(t),\quad
\left[\frac{\partial u}{\partial{\vec\nu}}\right]_{\Gamma(t)}
 = -\mathcal{V}\quad\!\! \text{on } \Gamma(t),\quad
\frac{\partial u}{\partial {\vec\nu_\Omega}} = 0 \quad\!\! \text{on } 
\partial\Omega.
\end{equation}
{From} now on we are going to restrict ourselves to a class of anisotropies
first proposed in \cite{triplejANI,ani3d}. That is, we assume that the
anisotropy can be written as 
\begin{equation} \label{eq:g}
\gamma(\vec p) = \left(
\sum_{\ell=1}^L [G_{\ell}\vec p \cdot \vec p]^{\frac r2}\right)^\frac1r,
\end{equation}
where $r\in[1,\infty)$ and $G_{\ell} \in \bR^{d\times d}$, 
$\ell=1,\ldots, L$, are symmetric and positive definite.
We also define $\tG_\ell = [\det G_\ell]^{\frac{1}{d-1}} G_\ell^{-1}$ for
$\ell=1,\ldots, L$.
Using a suitable differential calculus, the authors in \cite{ani3d} then
derived the following anisotropic analogue of \eqref{eq:varkappa}  
\begin{equation*} 
\left\langle \varkappa_\gamma \vec\nu,\vec\eta \right\rangle_{\Gamma(t)}
+ \left\langle 
\nabs^{\tG}\vec\id, \nabs^{\tG}\vec\eta \right\rangle_{\Gamma(t),\gamma}
=0 \quad\forall\ \vec\eta \in [H^1(\Gamma(t))]^d,
\end{equation*}
see \cite{ani3d} and also \cite[(110)]{bgnreview} for the precise definitions.
Hence the natural anisotropic analogue of \eqref{eq:WMS} is given by
\begin{subequations} \label{eq:aniWMS}
\begin{align} \label{eq:aniWMS1}
\left(\nabla u,\nabla\phi\right) -
\left\langle \vec{\mathcal{V}},\phi\vec\nu \right\rangle_{\Gamma(t)} & = 0
\qquad\forall\ \phi\in H^1(\Omega),\\
\left\langle u - \varkappa_\gamma,\chi \right\rangle_{\Gamma(t)} 
& = 0 \qquad\forall\ \chi\in L^2(\Gamma(t))\label{eq:aniWMS2},\\
\left\langle \varkappa_\gamma \vec\nu,\vec\eta \right\rangle_{\Gamma(t)}
+ \left\langle 
\nabs^{\tG}\vec\id, \nabs^{\tG}\vec\eta \right\rangle_{\Gamma(t),\gamma}
& =0 \qquad
\forall\ \vec\eta\in [H^1(\Gamma(t))]^d. \label{eq:aniWMS3}
\end{align}
\end{subequations}
The same testing procedure as in the isotropic setting shows that solutions to
\eqref{eq:aniWMS} satisfy
\begin{equation} \label{eq:anistruct}
\ddt |\Gamma(t)|_\gamma 
= - \left\langle \varkappa_\gamma, \mathcal{V} \right\rangle_{\Gamma(t)}
= -( \nabla u, \nabla u)\leq 0
\quad \text{and} \quad
\ddt \vol(\Omega_-(t)) = 0,
\end{equation}
where in the first equation we have noted Lemma~97 from \cite{bgnreview}. 

For the adaptation of \eqref{eq:fdMS}$^{(h)}$ to the anisotropic setting 
we make use of
the stable discretization of \eqref{eq:aniWMS3} introduced in \cite{ani3d}. To
this end, we define 
\begin{equation} \label{eq:ipGG}
\left\langle \nabs^{\tG_{\ell}}\vec X^{m+1},
\nabs^{\tG_{\ell}}\vec\eta \right\rangle_{\Gamma^m,\gamma} 
=\sum_{\ell=1}^L \int_{\Gamma^m} \left[
\frac{\gamma_\ell(\vec\nu^{m+1} \circ \vec X^{m+1})}
{\gamma(\vec\nu^{m+1}\circ \vec X^{m+1})} \right]^{r-1}
(\nabs^{\tG_{\ell}}\vec X^{m+1},\nabs^{\tG_{\ell}}\vec\eta)_{\tG_\ell}
\gamma_{\ell}(\vec\nu^m) \dH{d-1} 
\end{equation}
for $\Gamma^{m+1} = \vec X^{m+1}(\Gamma^m)$ with normal $\vec\nu^{m+1}$ and
$\vec X^{m+1}, \vec\eta \in \Vhm$. Here $\nabs^{\tG_{\ell}}$ is a surface
differential operator weighted by $\tG_\ell$, while $(\cdot,\cdot)_{\tG_\ell}$
denotes the inner product in $\bR^d$ induced by the symmetric positive definite
matrix $\tG_\ell$, see (108) and (111) in \cite{bgnreview} for details.
We note that \eqref{eq:ipGG} depends linearly on $\vec X^{m+1}$ if $r=1$.
Then our fully discrete approximation of \eqref{eq:aniMS} is given as follows.

Let the closed polyhedral hypersurface $\Gamma^0$ be an approximation of
$\Gamma(0)$. Then, for $m=0,\ldots,M-1$, find 
$(U^{m+1},\vec X^{m+1},\kappa^{m+1}_\gamma) \in S^m \times \Vhm \times \Whm$ 
and $\Gamma^{m+1} = \vec X^{m+1}(\Gamma^m)$ such that
\begin{subequations} \label{eq:fdaniMS}
\begin{align}\label{eq:fdaniMS1}
\left(\nabla U^{m+1}, \nabla\varphi\right) 
- \left\langle \pi_{\Gamma^m}\left[
\frac{\vec X^{m+1}-\vec\id}{\ttau_m} \cdot\vec\omega^{m+\frac12}\right],\varphi
\right\rangle_{\Gamma^m}^{(h)} & = 0 \qquad\forall\ \varphi \in S^m,\\
\left\langle U^{m+1},\chi\right\rangle_{\Gamma^m}^{(h)} 
-\left\langle\kappa^{m+1}_\gamma,\chi\right\rangle^h_{\Gamma^m} 
&= 0 \qquad\forall\ \chi \in \Whm ,\label{eq:fdaniMS2}\\
\left\langle \kappa^{m+1}_\gamma\vec\omega^{m+\frac12},
\vec\eta\right\rangle^h_{\Gamma^m}
+ \left\langle \nabs^{\tG_{\ell}}\vec X^{m+1},
\nabs^{\tG_{\ell}}\vec\eta \right\rangle_{\Gamma^m,\gamma}
&= 0
\qquad\forall\ \vec\eta \in \Vhm. \label{eq:fdaniMS3}
\end{align}
\end{subequations}
Once again, \eqref{eq:fdaniMS}$^{(h)}$ is a structure preserving 
approximation, in that
its solution satisfy discrete analogues of \eqref{eq:anistruct}. 

\begin{theorem} 
Let $(U^{m+1}, \vec X^{m+1},\kappa^{m+1}_\gamma)\in
S^m\times \Vhm\times \Whm$ be a solution to \eqref{eq:fdaniMS}$^{(h)}$. 
Then the enclosed volume is preserved, i.e.\
$\vol(\Omega_-^{m+1}) = \vol(\Omega_-^{m})$.
In addition, if $d=2$ or $d=3$, then the solution satisfies
the stability estimate
\begin{equation} \label{eq:thmanistab}
|\Gamma^{m+1}|_\gamma + \ttau_m \left(\nabla U^{m+1}, \nabla U^{m+1} \right)
\leq |\Gamma^m|_\gamma.
\end{equation}
\end{theorem}
\begin{proof} 
The volume preservation property follows as in the proof of
Theorem~\ref{thm:HDMS}, on choosing $\varphi=1$ in \eqref{eq:fdaniMS1}.
Similarly, for the discrete stability bound we choose
$\varphi= U^{m+1}$ in \eqref{eq:fdaniMS1}, 
$\chi = \pi_{\Gamma^m}[(\vec X^{m+1}-\vec\id)\cdot\vec\omega^{m+\frac12}]$ 
in \eqref{eq:fdaniMS2} and 
$\vec\eta = \vec X^{m+1}-\vec\id_{\mid_{\Gamma^m}}$ in \eqref{eq:fdaniMS3} 
in order to obtain
\begin{equation} \label{eq:anistep1}
\ttau_m\left(\nabla U^{m+1}, \nabla U^{m+1} \right) + 
\left\langle\nabs^{\tG_\ell}\vec X^{m+1},\nabs^{\tG_\ell}(\vec X^{m+1}-\vec\id) 
\right\rangle_{\Gamma^m,\gamma} =0.
\end{equation}
Now we recall from Lemma~102 in \cite{bgnreview} the result
\begin{equation} \label{eq:anistab2d3d}
\left\langle\nabs^{\tG_\ell}\vec X^{m+1},\nabs^{\tG_\ell}(\vec X^{m+1}-\vec\id) 
\right\rangle_{\Gamma^m,\gamma} \geq |\Gamma^{m+1}|_\gamma - |\Gamma^m|_\gamma
\end{equation}
for the cases $d=2$ and $d=3$. Combining \eqref{eq:anistep1} and
\eqref{eq:anistab2d3d} yields the desired result \eqref{eq:thmanistab}. 
\end{proof}

The adaptation of the iterative solution method \eqref{eq:itMS}, 
\eqref{eq:stop} to the anisotropic case is easy in the case $r=1$. 
For $r>1$ we combine the lagging of the nonlinear term $\vec\omega^{m+\frac12}$
in \eqref{eq:fdaniMS1} and \eqref{eq:fdaniMS3} 
with the lagging of $\vec\nu^{m+1}$ in the second term of \eqref{eq:fdaniMS3},
compare with \eqref{eq:ipGG}. Overall, we use the following iteration in order
to find a solution to \eqref{eq:fdaniMS}$^{(h)}$. 
For $i \geq 0$ find $(U^{m+1,i+1},\vec X^{m+1,i+1},\kappa^{m+1,i+1}_\gamma) 
\in S^m \times \Vhm \times \Whm$ 
such that
\begin{subequations} \label{eq:itaniMS}
\begin{align}\label{eq:itaniMS1}
&
\left(\nabla U^{m+1,i+1}, \nabla\varphi\right) 
- \left\langle \pi_{\Gamma^m}\left[
\frac{\vec X^{m+1,i+1}-\vec\id}{\ttau_m} \cdot\vec\omega^{m+\frac12,i}\right],
\varphi \right\rangle_{\Gamma^m}^{(h)} = 0 \qquad\forall\ \varphi \in S^m,\\
&\left\langle U^{m+1,i+1},\chi\right\rangle_{\Gamma^m}^{(h)} 
-\left\langle\kappa^{m+1,i+1}_\gamma,\chi\right\rangle^h_{\Gamma^m} 
= 0 \qquad\forall\ \chi \in \Whm ,\label{eq:itaniMS2}\\
&\left\langle \kappa^{m+1,i+1}_\gamma\vec\omega^{m+\frac12,i},
\vec\eta\right\rangle^h_{\Gamma^m} \nonumber \\ & \
+ \sum_{\ell=1}^L \int_{\Gamma^m} \left[
\frac{\gamma_\ell(\vec\nu^{m+1,i} \circ \vec X^{m+1,i})}
{\gamma(\vec\nu^{m+1,i}\circ \vec X^{m+1,i})} \right]^{r-1}
(\nabs^{\tG_{\ell}}\vec X^{m+1},\nabs^{\tG_{\ell}}\vec\eta)_{\tG_\ell}
\gamma_{\ell}(\vec\nu^m) \dH{d-1} 
= 0
\qquad\forall\ \vec\eta \in \Vhm \label{eq:itaniMS3}
\end{align}
\end{subequations}
and set $\Gamma^{m+1,i+1} = \vec X^{m+1,i+1}(\Gamma^m)$. The iteration is
stopped when the criterion \eqref{eq:stop} is satisfied. We note that the
second term in \eqref{eq:itaniMS3} is a linearization of \eqref{eq:ipGG}.
The term will be independent of $\vec X^{m+1,i}$ in the case $r=1$.

\setcounter{equation}{0}
\section{Numerical results} \label{sec:nr}

We implemented the fully discrete finite element approximations 
\eqref{eq:DMS}$^{(h)}$, \eqref{eq:fdMS}$^{(h)}$ and 
\eqref{eq:fdaniMS}$^{(h)}$ within the
finite element toolbox ALBERTA, see \cite{Alberta}. The systems of
linear equations arising from \eqref{eq:DMS}$^{(h)}$, \eqref{eq:itMS}$^{(h)}$ 
and \eqref{eq:itaniMS}$^{(h)}$, in the case $d=2$, 
are solved with the help of the sparse 
factorization package UMFPACK, see \cite{Davis04}. For the simulations in 3d,
on the other hand,
we employ the Schur complement solver described in \cite[(4.9)]{dendritic}.
For the stopping criterion in \eqref{eq:stop} we use the value
$\text{tol}=10^{-10}$.

For the triangulation $\mathcal{T}^m$ of the bulk domain $\Omega$, that is used
for the bulk finite element space $S^m$, we use an adaptive mesh that uses 
fine elements close to the interface $\Gamma^m$ and coarser elements away 
from it. The precise strategy is as described in \cite[\S5.1]{dendritic} 
and for a domain $\Omega=(-H,H)^d$ and two integer parameters
$N_c < N_f$ results in elements with maximal diameter approximately equal to
$h_f = \frac{2H}{N_f}$ close to $\Gamma^m$ and elements with maximal diameter
approximately equal to $h_c = \frac{2H}{N_c}$ far away from it. 
For all our computations we use $H=4$.
An example adaptive mesh is shown in Figure~\ref{fig:2dmesh}, below.

We stress that due to the unfitted nature of our finite element approximations,
special quadrature rules need to be employed in order to assemble terms that
feature both bulk and surface finite element functions. 
An example is the first term in \eqref{eq:fdMS2}.
For the schemes using numerical integration, e.g.\ \eqref{eq:fdMS}$^h$, this
task boils down to finding for each vertex of $\Gamma^m$ the bulk element
$\sigmaO^m \in\mathcal{T}^m$ it resides in, together with its barycentric
coordinates with respect to that bulk element.
For the remaining schemes that task is more involved.
Then the most challenging aspect of assembling the contributions for e.g.\
the first term in \eqref{eq:fdMS2}, for the scheme \eqref{eq:fdMS},
is to compute intersections
$\sigma^m \cap \sigmaO^m$ between an arbitrary surface element
$\sigma^m \subset \Gamma^m$ and an element $\sigmaO^m\in\mathcal{T}^m$
of the bulk mesh.
An algorithm that describes how these intersections can be calculated is given
in \cite[p.\ 6284]{dendritic}, see also Figure~4 in 
\rm \cite{dendritic} for a visualization of possible intersections of the form
$\sigma^m \cap \sigmaO^m$ in $\bR^3$.

Throughout this section we use (almost) uniform time steps, in that
$\ttau_m=\ttau$ for $m=0,\ldots, M-2$ and 
$\ttau_{M-1} = T - t_{m-1} \leq \ttau$. For many of the presented simulations we
will put particular emphasis on the volume preserving aspect, and so we recall
that given a polyhedral surface $\Gamma^m$, the enclosed volume can be computed
by
\begin{equation} \label{eq:fdvol}
\vol(\Omega^m_-) = 
\frac1d \int_{\Gamma^m} \vec\id \cdot \vec\nu^m \dH{d-1} ,
\end{equation}
where we have used the divergence theorem. We note that the integrand in
\eqref{eq:fdvol} is piecewise constant on $\Gamma^m$. For later use we also
define the relative volume loss at time $t=t_m$ as 
\[
v_\Delta^m = \frac{\vol(\Omega^0_-) - \vol(\Omega^m_-)}{\vol(\Omega^0_-)}.
\]

\subsection{Convergence experiment}

We begin with a convergence experiment for the scheme \eqref{eq:fdMS}
for the cases $d=2$ and $d=3$. To this
end, we recall from \cite[\S6.6]{dendritic} the following exact solution to
\eqref{eq:MS} consisting of two concentric spheres. 
Let $(\Gamma(t))_{t\in[0,T]}$ be a solution of \eqref{eq:MS}, where
$\Gamma(t) = \partial\Omega_-(t)$ with
$\Omega_-(t) = \{ \vec z \in \bR^3 : r_1(t) < |\vec z| < r_2(t) \}$.
Then the two radii $r_1 < r_2$ satisfy the following system of nonlinear ODEs:
In the case $d=2$ we have
\begin{subequations} \label{eq:ODE}
\begin{equation}
[r_1]_t = - \frac1{r_1}\frac{\frac1{r_1} + \frac1{r_2}}{\ln\frac{r_2}{r_1}} 
\quad\mbox{ and }\quad
[r_2]_t 
= \frac{r_1}{r_2}[r_1]_t \qquad \forall\ t \in[0,T_0), \label{eq:ODEa}
\end{equation}
while for $d=3$ it holds that
\begin{equation}
[r_1]_t = - \frac2{r_1^2}\frac{r_1 + r_2}{r_2-r_1} \quad\mbox{ and }\quad
[r_2]_t 
= \frac{r_1^2}{r_2^2}[r_1]_t \qquad \forall\ t \in[0,T_0), \label{eq:ODEb}
\end{equation}
\end{subequations}
where $T_0$ is the extinction time of the smaller sphere, 
i.e.\ $\lim_{t\to T_0} r_1(t) = 0$.
The corresponding solution $u$ satisfying \eqref{eq:MS} is given
by the radially symmetric function
\begin{equation}
u(\vec{z},t) = \begin{cases}
 -\frac{d-1}{r_2(t)} & |\vec{z}| \geq r_2(t), \\
\begin{cases}
\frac1{r_1(t)} - \ln\frac{|\vec{z}|}{r_1(t)}
\dfrac{\frac1{r_1(t)} + \frac1{r_2(t)}}{\ln\frac{r_2(t)}{r_1(t)}} & d = 2 \\
-\frac4{r_2(t) - r_1(t)} + 
\frac2{|\vec{z}|}\frac{{r_1(t)} + {r_2(t)}}{r_2(t)-r_1(t)} & d = 3
\end{cases}
& r_1(t) \leq |\vec{z}| \leq r_2(t), \\
 \frac{d-1}{r_1(t)} & |\vec{z}| \leq r_1(t).
\end{cases} 
\label{eq:ODE_u}
\end{equation}
The volume preserving property of the flow implies that
$v(t) = r_2^d(t) - r_1^d(t)$ is an invariant, so that
$r_2(t) = (v(0) + r_1^d(t))^\frac1d$. Hence $r_1$ satisfies
\begin{equation}
[r_1]_t = \begin{cases} 
- \dfrac1{r_1}\dfrac{\frac1{r_1} + (v(0) + r_1^2)^{-\frac12}}
{\ln\frac{(v(0) + r_1^2)^{\frac12}}{r_1}} & d = 2,  \\[5mm]
- \dfrac2{r_1^2}\dfrac{r_1 + (v(0) + r_1^3)^{\frac13}}
{(v(0) + r_1^3)^{\frac13}-r_1} & d = 3 ,
\end{cases}
\qquad \forall\ t \in[0,T_0).
\label{eq:ODE1d}
\end{equation}
In order to obtain a higher accuracy for the reference solution 
in our numerical convergence experiments, 
rather than integrating \eqref{eq:ODE1d} directly, 
we rather use a root-finding algorithm for the equation
\begin{equation*} 
0 = t + \begin{cases} \displaystyle
\int_{r_1(0)}^{r_1(t)} r\frac{\ln\frac{(v(0) + r^2)^{\frac12}}{r}}
{\frac1{r} + (v(0) + r^2)^{-\frac12}} \;{\rm d}r & d = 2 ,\\[5mm]
\displaystyle
\int_{r_1(0)}^{r_1(t)} \frac{r^2}2\frac{(v(0) + r^3)^{\frac13}-r}
{r + (v(0) + r^3)^{\frac13}} \;{\rm d}r & d = 3  ,
\end{cases}
\qquad \forall\ t \in[0,T_0)
\end{equation*}
in order to find $r_1(t)$.

For the initial radii $r_1(0) = 2.5$, $r_2(0) = 3$ and 
the time interval $[0,T]$ with $T=\frac12$, 
so that $r_1(T) \approx 1.66$ and $r_2(T) \approx 2.35$, 
we perform a convergence 
experiment for the true solution \eqref{eq:ODE}, at first for $d=2$.
To this end, for $i=0\to 4$, we set 
$N_f = \frac12 K = 2^{7+i}$, $N_c = 4^{i}$
and $\tau= 4^{3-i}\times10^{-3}$. 
We visualize the evolution with the help of the discrete solutions computed
with the scheme \eqref{eq:fdMS} for the run $i=1$ in Figure~\ref{fig:2dmesh},
where we also present a plot of the final bulk mesh $\mathcal{T}^M$ in order 
to show the effect of the adaptive mesh refinement strategy.
\begin{figure}
\center
\includegraphics[angle=-90,width=0.3\textwidth]{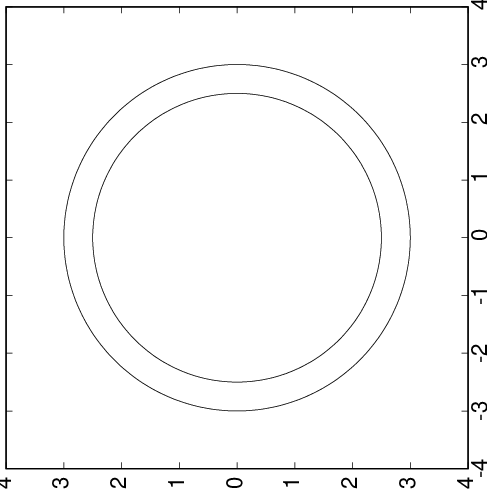}
\includegraphics[angle=-90,width=0.3\textwidth]{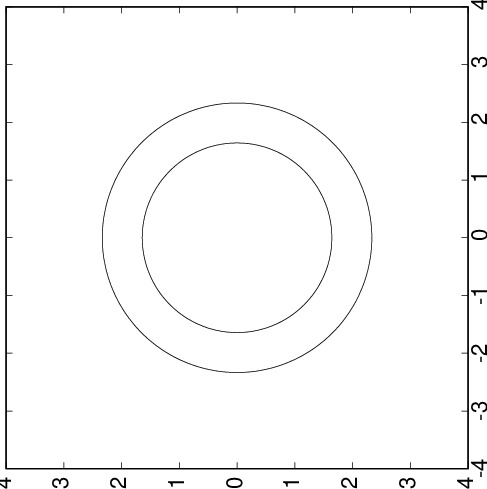}\quad
\includegraphics[angle=-90,width=0.3\textwidth]{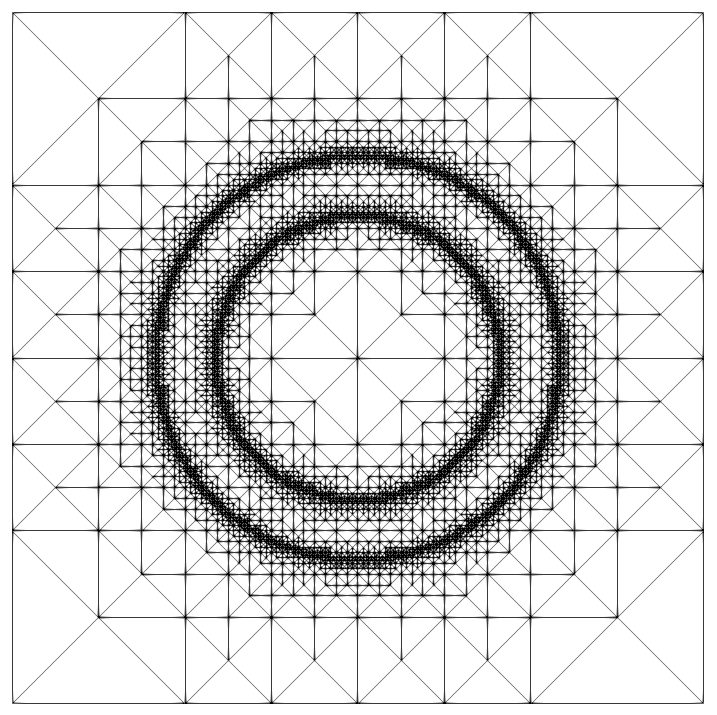}
\caption{The solution \eqref{eq:ODE} at times $t=0$ and $t=\frac12$, 
as well as the adaptive bulk mesh $\mathcal{T}^M$.}
\label{fig:2dmesh}
\end{figure}%
In Table~\ref{tab:fdMS2dT05} we display the errors 
\[\errorXx = \max_{m=1,\ldots, M} 
\max_{k=1,\ldots, K} \dist(\vec{q}^m_k, \Gamma(t_m))
\]
and
\[
\errorUu = \max_{m=1,\ldots, M}\|U^m - I^mu(\cdot,t_m)\|_{L^\infty(\Omega)},
\]
where $I^m : C^0(\overline\Omega) \to S^m$ denotes the standard 
interpolation operator. We also let $K_\Omega^m$ denote the number of
degrees of freedom of $S^m$, and define $h^m_\Gamma = \max_{j = 1,\ldots,J} 
\diam(\sigma^m_j)$.
As a comparison, we show the same error computations for the linear scheme
\eqref{eq:DMS} in Table~\ref{tab:DMS2dT05}. 
As expected, we observe true volume preservation for the scheme \eqref{eq:fdMS} 
in Table~\ref{tab:fdMS2dT05}, up to solver tolerance, while the relative volume
loss in Table~\ref{tab:DMS2dT05} decreases as $\ttau$ becomes smaller. 
Surprisingly, the two error quantities $\errorXx$ and $\errorUu$ are generally
lower in Table~\ref{tab:DMS2dT05} compared to Table~\ref{tab:fdMS2dT05},
although the difference becomes smaller with smaller discretization parameters.
For completeness, we also present the errors for the same convergence
experiment for the two schemes \eqref{eq:fdMS}$^h$ and \eqref{eq:DMS}$^h$ with
numerical integration, see Tables~\ref{tab:numintfdMS2dT05} and 
\ref{tab:numintDMS2dT05}.
\begin{table}
\center
\begin{tabular}{c|c|c|c|c|c|c}
 $h_{f}$ & $h^M_\Gamma$ & $\errorUu$ & $\errorXx$ & $K^M_\Omega$ & $K$
 & $|v_\Delta^M|$ \\ \hline
 6.2500e-02 & 1.1400e-01 & 1.5609e-01 & 3.4036e-02 & 2925 & 256 & $<10^{-10}$\\
 3.1250e-02 & 5.7282e-02 & 4.5306e-02 & 1.7416e-02 & 5101 & 512 & $<10^{-10}$\\
 1.5625e-02 & 2.8714e-02 & 1.4406e-02 & 8.9079e-03 & 9785 & 1024& $<10^{-10}$\\
 7.8125e-03 & 1.4375e-02 & 5.0773e-03 & 4.6020e-03 & 21557 & 2048& $<10^{-10}$\\
 3.9062e-03 & 7.1929e-03 & 2.8734e-03 & 2.1860e-03 & 96781 & 4096& $<10^{-10}$\\
\end{tabular}
\caption{Convergence test for \eqref{eq:ODE} over the time interval 
$[0,\frac12]$ for the scheme \eqref{eq:fdMS}.}
\label{tab:fdMS2dT05}
\end{table}%
\begin{table}
\center
\begin{tabular}{c|c|c|c|c|c|c}
 $h_{f}$ & $h^M_\Gamma$ & $\errorUu$ & $\errorXx$ & $K^M_\Omega$ & $K$
 & $|v_\Delta^M|$ \\ \hline
 6.2500e-02 & 1.1497e-01 & 1.4990e-01 & 5.1377e-03 & 2869 & 256 & 1.2e-02\\
 3.1250e-02 & 5.7408e-02 & 4.3367e-02 & 7.7591e-03 & 5097 & 512 & 3.2e-03\\
 1.5625e-02 & 2.8730e-02 & 1.3917e-02 & 6.4656e-03 & 9857 & 1024 & 8.3e-04\\
 7.8125e-03 & 1.4377e-02 & 4.9546e-03 & 3.9948e-03 & 21593 & 2048 & 2.1e-04\\
 3.9062e-03 & 7.1932e-03 & 2.7345e-03 & 2.0351e-03 & 96969 & 4096 & 5.1e-05\\
\end{tabular}
\caption{Convergence test for \eqref{eq:ODE} over the time interval 
$[0,\frac12]$ for the scheme \eqref{eq:DMS}.}
\label{tab:DMS2dT05}
\end{table}%
\begin{table}
\center
\begin{tabular}{c|c|c|c|c|c|c}
 $h_{f}$ & $h^M_\Gamma$ & $\errorUu$ & $\errorXx$ & $K^M_\Omega$ & $K$
 & $|v_\Delta^M|$ \\ \hline
 6.2500e-02 & 1.1433e-01 & 1.6079e-01 & 2.4789e-02 & 2941 & 256 & $<10^{-10}$\\
 3.1250e-02 & 5.7357e-02 & 4.9133e-02 & 1.3107e-02 & 5077 & 512 & $<10^{-10}$\\
 1.5625e-02 & 2.8733e-02 & 1.6422e-02 & 6.8358e-03 & 9865 & 1024& $<10^{-10}$\\
 7.8125e-03 & 1.4380e-02 & 6.1040e-03 & 3.5755e-03 & 21605 & 2048& $<10^{-10}$\\
 3.9062e-03 & 7.1941e-03 & 2.6860e-03 & 1.6743e-03 & 96893 & 4096& $<10^{-10}$\\
\end{tabular}
\caption{Convergence test for \eqref{eq:ODE} over the time interval 
$[0,\frac12]$ for the scheme \eqref{eq:fdMS}$^h$.}
\label{tab:numintfdMS2dT05}
\end{table}%
\begin{table}
\center
\begin{tabular}{c|c|c|c|c|c|c}
 $h_{f}$ & $h^M_\Gamma$ & $\errorUu$ & $\errorXx$ & $K^M_\Omega$ & $K$
 & $|v_\Delta^M|$ \\ \hline
 6.2500e-02 & 1.1530e-01 & 1.6291e-01 & 1.3785e-02 & 2881 & 256 & 1.2e-02\\
 3.1250e-02 & 5.7482e-02 & 4.7307e-02 & 4.5358e-03 & 5185 & 512 & 3.2e-03\\
 1.5625e-02 & 2.8749e-02 & 1.5926e-02 & 4.4224e-03 & 9757 & 1024 & 8.2e-04\\
 7.8125e-03 & 1.4382e-02 & 5.9809e-03 & 2.9693e-03 & 21501 & 2048 & 2.1e-04\\
 3.9062e-03 & 7.1943e-03 & 2.5431e-03 & 1.5237e-03 & 96997 & 4096 & 5.1e-05\\
\end{tabular}
\caption{Convergence test for \eqref{eq:ODE} over the time interval 
$[0,\frac12]$ for the scheme \eqref{eq:DMS}$^h$.}
\label{tab:numintDMS2dT05}
\end{table}%

We also perform a convergence experiment for the true solution
\eqref{eq:ODE} for $d=3$. To this end, we choose
the initial radii $r_1(0) = 2.5$, $r_2(0) = 3$ and 
the time interval $[0,T]$ with $T=0.1$, 
so that $r_1(T) \approx 2.15$ and $r_2(T) \approx 2.77$.
Moreover, for $i=0\to 3$, we set $N_f = 2^{5+i}$, $N_c = 4^{i}$,
$\frac12 K=\hat K(i)$, where $(\hat K(0), \hat K(1), \hat K(2), 
\hat K(3)) = 
(770, 3074, 12290, 49154)$, and $\tau= 4^{3-i}\times10^{-3}$. 
The errors $\errorUu$ and $\errorXx$ for the four schemes  
\eqref{eq:fdMS}, \eqref{eq:fdMS}$^{(h)}$
\eqref{eq:DMS} and \eqref{eq:DMS}$^{(h)}$ 
on the interval $[0,T]$ with $T=0.1$ are displayed in 
Tables~\ref{tab:fdMS3d}, \ref{tab:numintfdMS3d}, \ref{tab:DMS3d} and
\ref{tab:numintDMS3d}.
\begin{table}
\center
\begin{tabular}{c|c|c|c|c|c|c}
 $h_{f}$ & $h^M_\Gamma$ & $\errorUu$ & $\errorXx$ & $K^M_\Omega$ & $K$ 
 & $|v_\Delta^M|$ \\ \hline
 2.5000e-01 & 5.6320e-01 & 7.3514e-01 & 1.3667e-01 & 10831 & 1540& $<10^{-10}$\\
 1.2500e-01 & 2.8759e-01 & 2.5135e-01 & 4.6999e-02 & 46311 & 6148& $<10^{-10}$\\
 6.2500e-02 & 1.4473e-01 & 9.1052e-02 & 1.9356e-02 & 188389 &24580&$<10^{-10}$\\
 3.1250e-02 & 7.2527e-02 & 3.5851e-02 & 8.7870e-03 & 956293 &98308&$<10^{-10}$\\
\end{tabular}
\caption{Convergence test for \eqref{eq:ODE} over the time interval $[0,0.1]$
for the scheme \eqref{eq:fdMS}.}
\label{tab:fdMS3d}
\end{table}%
\begin{table}
\center
\begin{tabular}{c|c|c|c|c|c|c}
 $h_{f}$ & $h^M_\Gamma$ & $\errorUu$ & $\errorXx$ & $K^M_\Omega$ & $K$ 
 & $|v_\Delta^M|$ \\ \hline
 2.5000e-01 & 5.6594e-01 & 8.9355e-01 & 1.3062e-01 & 10879 & 1540& $<10^{-10}$\\
 1.2500e-01 & 2.8815e-01 & 3.1381e-01 & 4.3354e-02 & 46335 & 6148& $<10^{-10}$\\
 6.2500e-02 & 1.4484e-01 & 1.2228e-01 & 1.7321e-02 & 188725 &24580&$<10^{-10}$\\
 3.1250e-02 & 7.2548e-02 & 5.7925e-02 & 7.6589e-03 & 970477 & 98308&$<10^{-10}$
\end{tabular}
\caption{Convergence test for \eqref{eq:ODE} over the time interval $[0,0.1]$
for the scheme \eqref{eq:fdMS}$^h$.}
\label{tab:numintfdMS3d}
\end{table}%
\begin{table}
\center
\begin{tabular}{c|c|c|c|c|c|c}
 $h_{f}$ & $h^M_\Gamma$ & $\errorUu$ & $\errorXx$ & $K^M_\Omega$ & $K$ 
 & $|v_\Delta^M|$ \\ \hline
 2.5000e-01 & 5.7042e-01 & 6.5892e-01 & 6.2158e-02 & 10879 & 1540 & 2.3e-02\\
 1.2500e-01 & 2.8847e-01 & 2.3273e-01 & 3.0705e-02 & 46375 & 6148 & 6.2e-03\\
 6.2500e-02 & 1.4485e-01 & 8.6575e-02 & 1.5551e-02 & 188725 & 24580 & 1.5e-03\\
 3.1250e-02 & 7.2548e-02 & 3.4759e-02 & 7.8760e-03 & 956293 & 98308 & 3.6e-04\\
\end{tabular}
\caption{Convergence test for \eqref{eq:ODE} over the time interval $[0,0.1]$
for the scheme \eqref{eq:DMS}.}
\label{tab:DMS3d}
\end{table}%
\begin{table}
\center
\begin{tabular}{c|c|c|c|c|c|c}
 $h_{f}$ & $h^M_\Gamma$ & $\errorUu$ & $\errorXx$ & $K^M_\Omega$ & $K$ 
 & $|v_\Delta^M|$ \\ \hline
 2.5000e-01 & 5.7401e-01 & 7.8420e-01 & 6.5619e-02 & 10879 & 1540 & 2.2e-02\\
 1.2500e-01 & 2.8908e-01 & 2.9887e-01 & 2.7440e-02 & 46423 & 6148 & 6.0e-03\\
 6.2500e-02 & 1.4497e-01 & 1.1943e-01 & 1.3544e-02 & 188965 & 24580 & 1.5e-03\\
 3.1250e-02 & 7.2572e-02 & 5.7009e-02 & 6.8226e-03 & 956821 & 98308 & 3.6e-04\\
\end{tabular}
\caption{Convergence test for \eqref{eq:ODE} over the time interval $[0,0.1]$
for the scheme \eqref{eq:DMS}$^h$.}
\label{tab:numintDMS3d}
\end{table}%
Similarly to the convergence experiments in 2d, we note that for the schemes
\eqref{eq:DMS}$^{(h)}$ the relative volume loss converges to zero as the
discretization parameters get smaller, while the
schemes \eqref{eq:fdMS}$^{(h)}$ preserve the volume exactly in every case.
The error quantities $\errorUu$ and $\errorXx$ behave very similarly for all
four schemes.

\subsection{Simulations in 2d}

In this subsection we consider some numerical experiments for the case $d=2$. 
In the first computation, we numerically confirm the well-known result shown in
\cite{Mayer98}, which says that the Mullins--Sekerka flow \eqref{eq:MS} 
does not preserve convexity. To this end, we choose for $\Gamma(0)$ an
elongated cigar shape of total dimension $7\times1$. The discretization
parameters for the computation are $N_f=128$, $N_c=16$, $\ttau=10^{-3}$, 
$T=2$ and $K=256$, and the results are shown in
Figure~\ref{fig:nonconvex2d}.
We observe that during the evolution the interface becomes nonconvex, before
reaching a circular steady state. As expected, the enclosed volume is preserved
during the evolution. This is not the case when using the scheme 
\eqref{eq:DMS}, as can be seen from Figure~\ref{fig:oldnonconvex2d}, where for
completeness we show the same simulation for this alternative finite element
approximation.
\begin{figure}
\center
\includegraphics[angle=-90,width=0.3\textwidth]{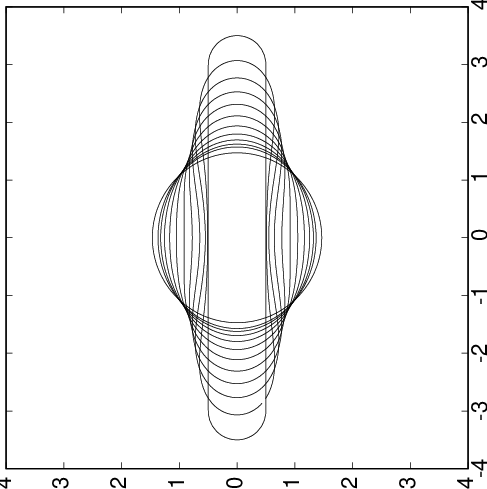}
\includegraphics[angle=-90,width=0.3\textwidth]{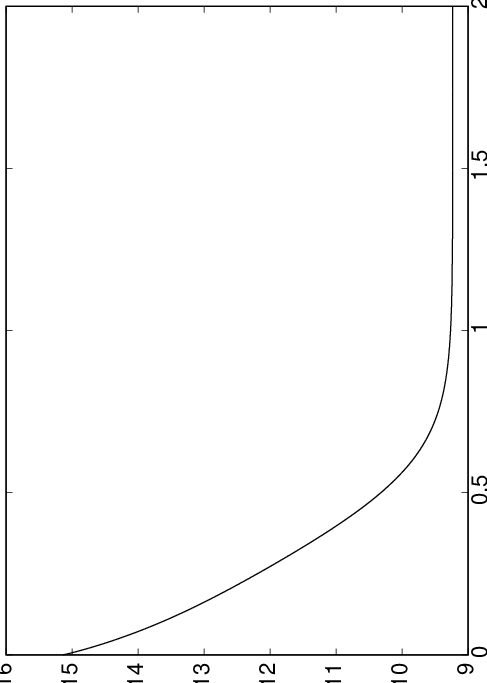}\quad
\includegraphics[angle=-90,width=0.3\textwidth]{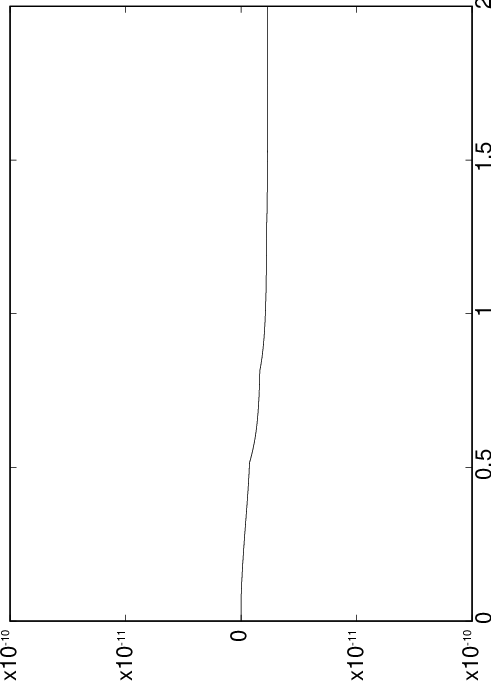}
\caption{$\Gamma^m$ at times $t=0,0.1,\ldots,1,T=2$
for the scheme \eqref{eq:fdMS}. We also show a plot of
the discrete energy $|\Gamma^m|$ and of the relative volume loss 
$v_\Delta^m$ over time.}
\label{fig:nonconvex2d}
\end{figure}%
\begin{figure}
\center
\includegraphics[angle=-90,width=0.3\textwidth]{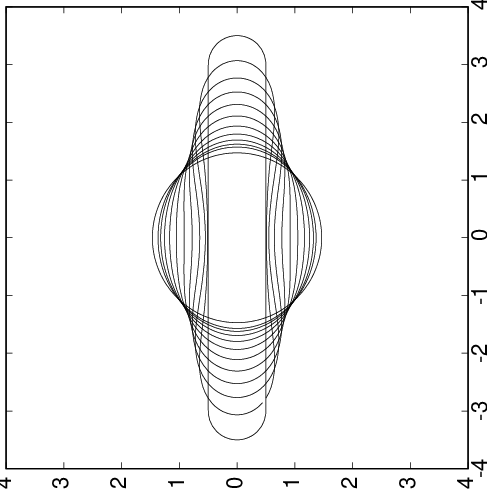}
\includegraphics[angle=-90,width=0.3\textwidth]{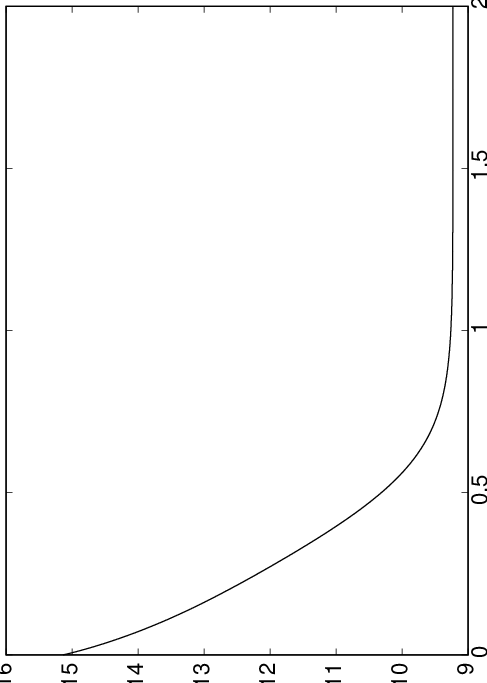}\quad
\includegraphics[angle=-90,width=0.3\textwidth]{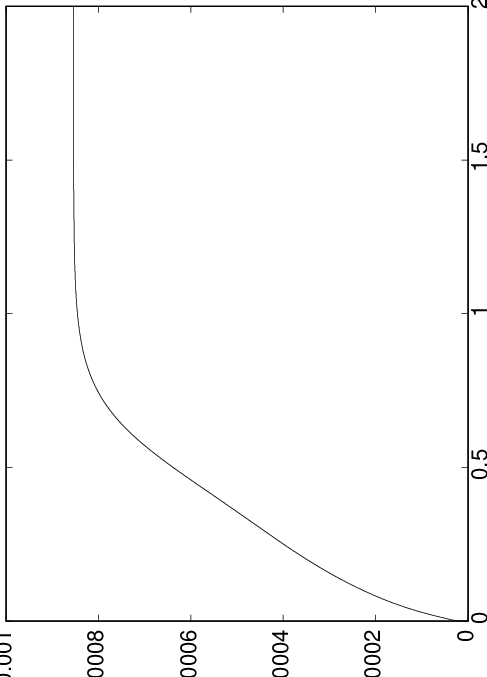}
\caption{$\Gamma^m$ at times $t=0,0.1,\ldots,1,T=2$
for the scheme \eqref{eq:DMS}. We also show a plot of
the discrete energy $|\Gamma^m|$ and of the relative volume loss 
$v_\Delta^m$ over time.}
\label{fig:oldnonconvex2d}
\end{figure}%

Our second simulation is for an anisotropic surface energy. Here we make use of
the fact that anisotropies of the form \eqref{eq:g} can be used to approximate
crystalline surface energies, where the isoperimetric minimizers 
(the so-called Wulff shapes) exhibit flat
parts and sharp corners. In particular, we choose the density
\begin{equation} \label{eq:gammabgnL}
\gamma_0(p) 
= \tfrac14 \sum_{\ell=1}^4 \sqrt{ [(R(\tfrac\pi{4})^\ell]^T D(\delta) 
(R(\tfrac\pi{4}))^\ell p \cdot p}, \quad \delta = 10^{-4},
\end{equation}
where 
$R(\theta)=\binom{\phantom{-}\cos\theta\ \sin\theta}{-\sin\theta\ \cos\theta}$
and $D(\delta) = \diag(1,\delta^2)$.
Then, inspired by the initial curve from \cite[Fig.~0]{AlmgrenT95},
see also \cite[Fig.~7]{finsler}, 
we perform a computation for our scheme \eqref{eq:fdaniMS}. We observe that
all the facets of the initial data are aligned with the Wulff shape of
\eqref{eq:gammabgnL} with $\delta=0$, i.e.\ \ regular octagon.
For the computations shown in Figure~\ref{fig:AlmgrenT95} we employed the 
discretization parameters $N_f = 256$, $N_c = 32$, $K=512$ and
$\Delta t = 10^{-3}$. 
We note that during the evolution all the facets remain aligned with the facets
of the Wulff shape. Some facets grow at the expense of others, leading to some
facets vanishing completely. Eventually a scaled Wulff shape is approached as a
steady state of the flow.
As a comparison, we also show the evolution for the isotropic case for the same
initial data, in Figure~\ref{fig:isoAlmgrenT95}. Here the nonconvex initial
data soon evolves to a convex curve, which then converges towards a circle.
\begin{figure}
\center
\includegraphics[angle=-90,width=0.3\textwidth]{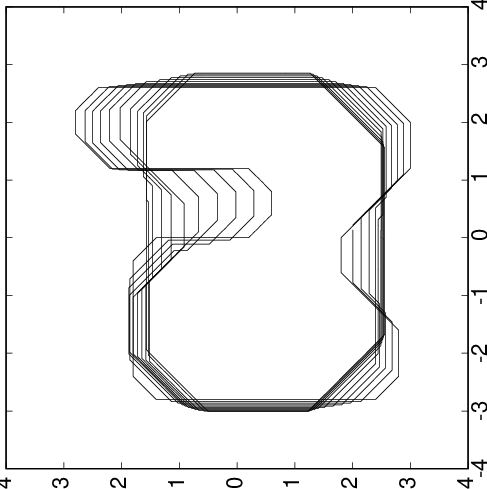}
\includegraphics[angle=-90,width=0.3\textwidth]{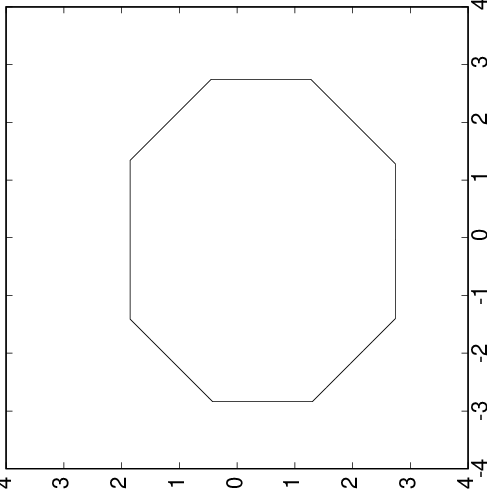}
\includegraphics[angle=-90,width=0.3\textwidth]{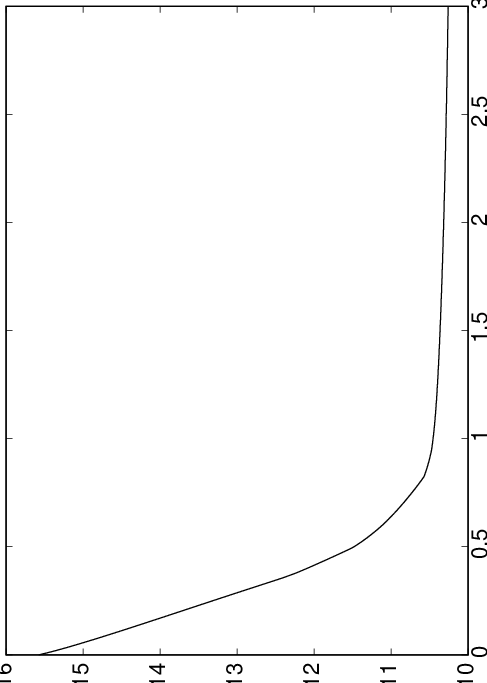}
\caption{$\Gamma^m$ at times $t=0,0.1,\ldots,1$, and at time $t=T=3$,
for the scheme \eqref{eq:fdaniMS}. 
We also show a plot of the discrete energy $|\Gamma^m|_\gamma$ 
over time.}
\label{fig:AlmgrenT95}
\end{figure}%
\begin{figure}
\center
\includegraphics[angle=-90,width=0.3\textwidth]{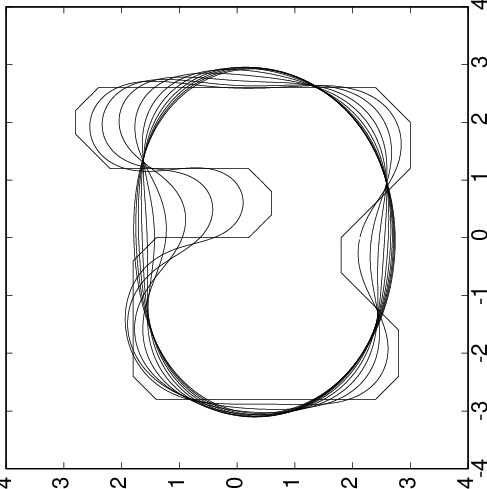}
\includegraphics[angle=-90,width=0.3\textwidth]{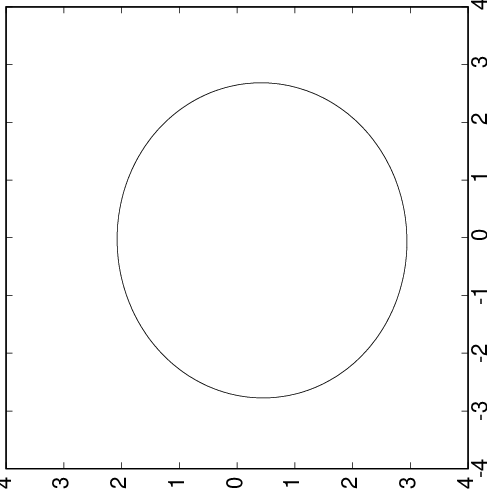}
\includegraphics[angle=-90,width=0.3\textwidth]{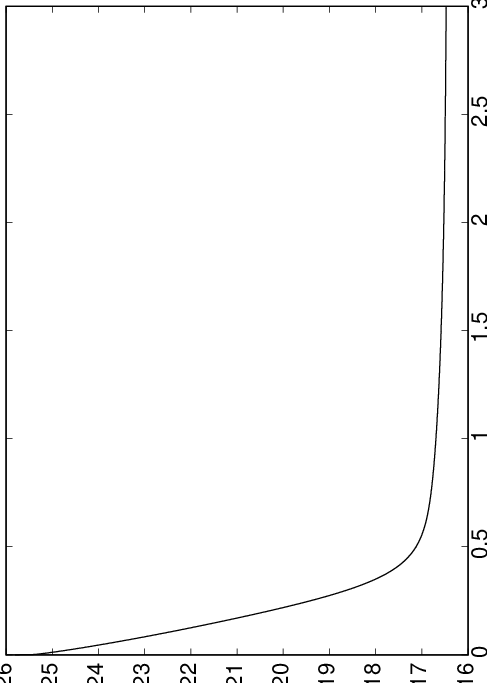}
\caption{$\Gamma^m$ at times $t=0,0.1,\ldots,1$, and at time $t=T=3$,
for the scheme \eqref{eq:fdMS}. 
We also show a plot of the discrete energy $|\Gamma^m|$ 
over time.}
\label{fig:isoAlmgrenT95}
\end{figure}%

\subsection{Simulations in 3d}

We end this section with some numerical simulations for the case $d=3$. All the
initial data will always be chosen symmetric with respect to the origin.
First we look at the 3d analogue of the experiment in
Figure~\ref{fig:nonconvex2d}, that is we start with an initial interface in the
shape of a rounded cylinder with total dimensions $7\times1\times1$.
The discretization parameters for this computation are $N_f=128$, $N_c=16$,
$\tau=10^{-3}$, $T=2$ and $K=1154$. 
\begin{figure}
\center
\includegraphics[angle=-0,width=0.18\textwidth]{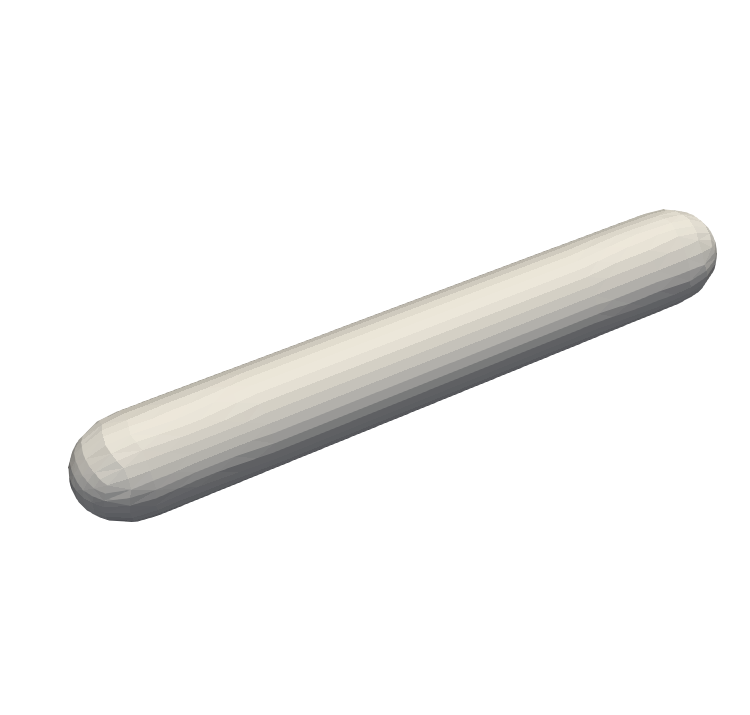}
\includegraphics[angle=-0,width=0.18\textwidth]{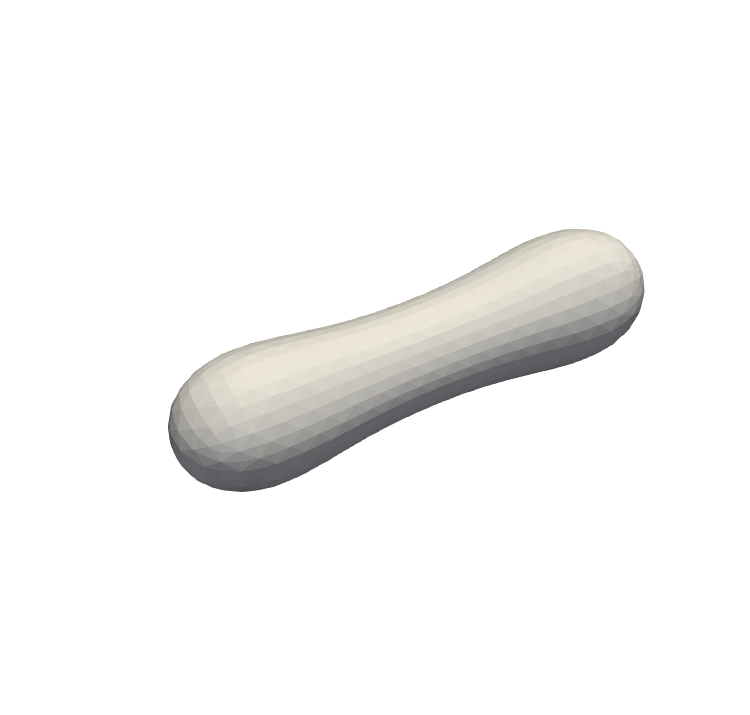}
\includegraphics[angle=-0,width=0.18\textwidth]{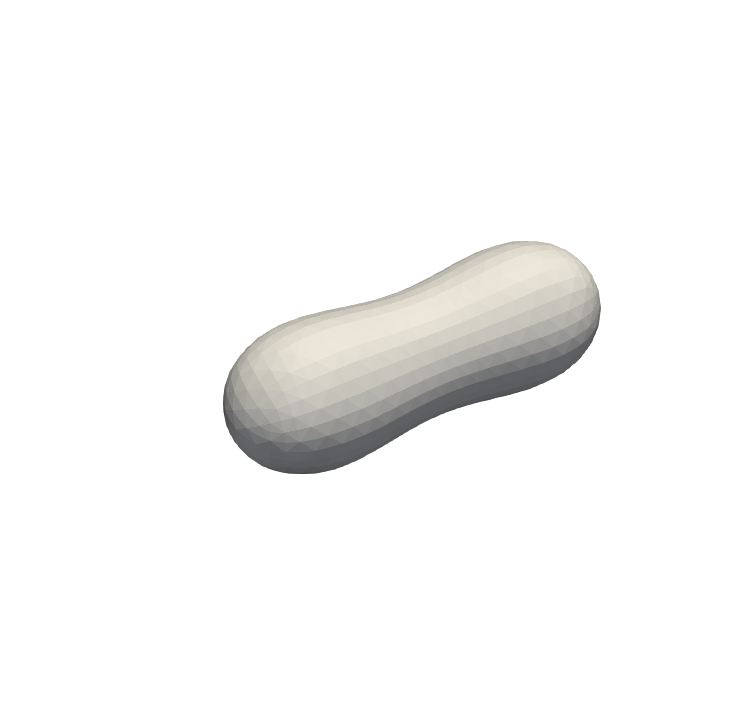}
\includegraphics[angle=-0,width=0.18\textwidth]{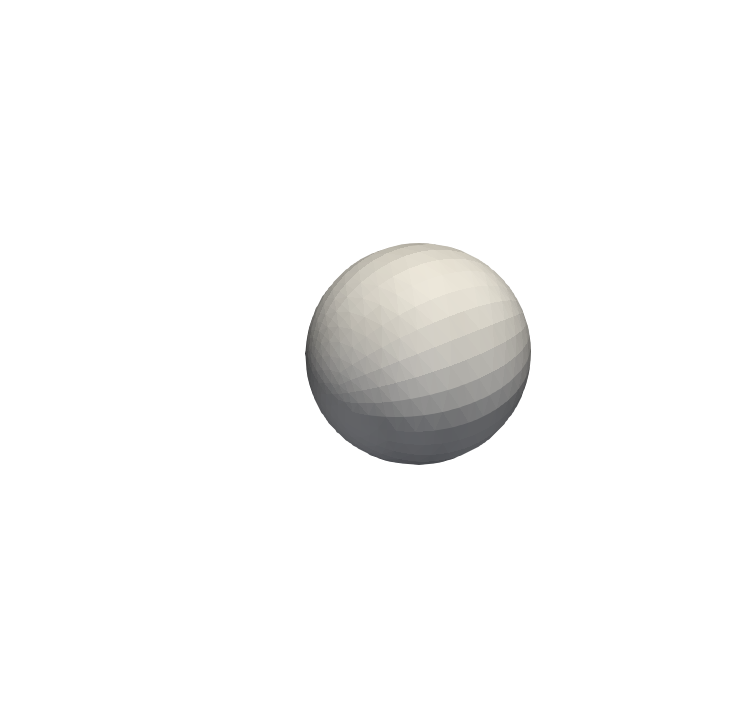}
\includegraphics[angle=-0,width=0.18\textwidth]{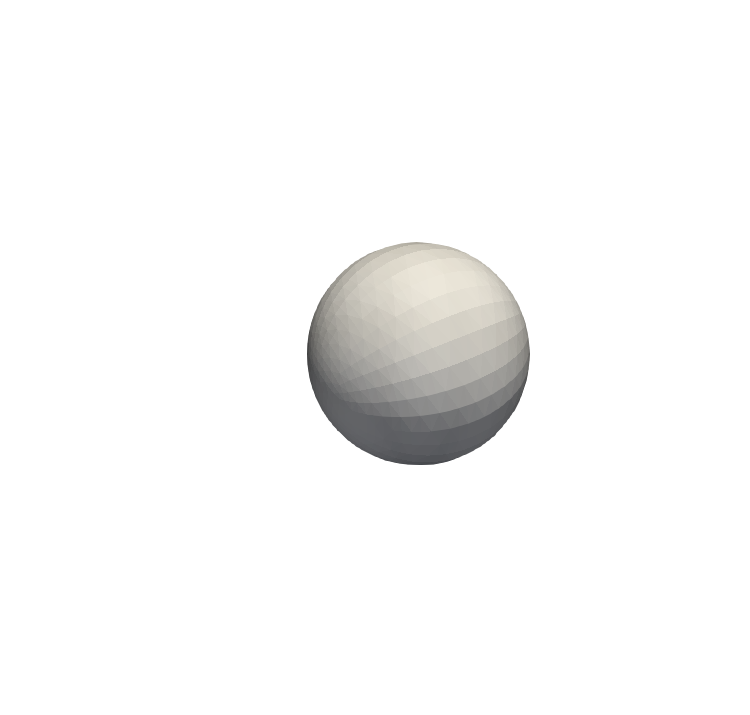}
\\
\includegraphics[angle=-90,width=0.3\textwidth]{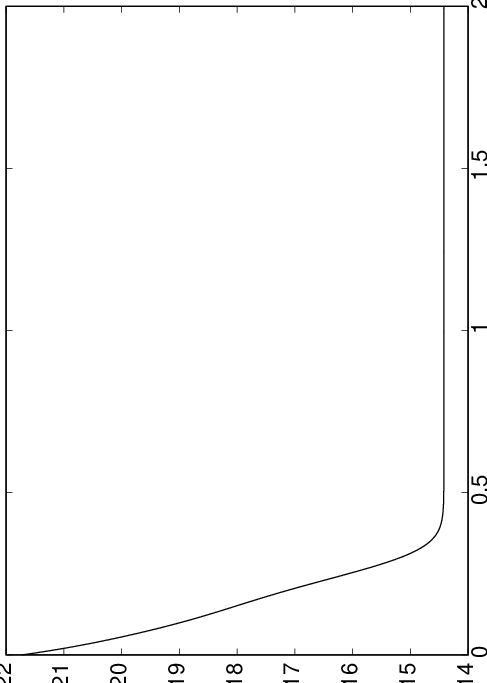}\quad
\includegraphics[angle=-90,width=0.3\textwidth]{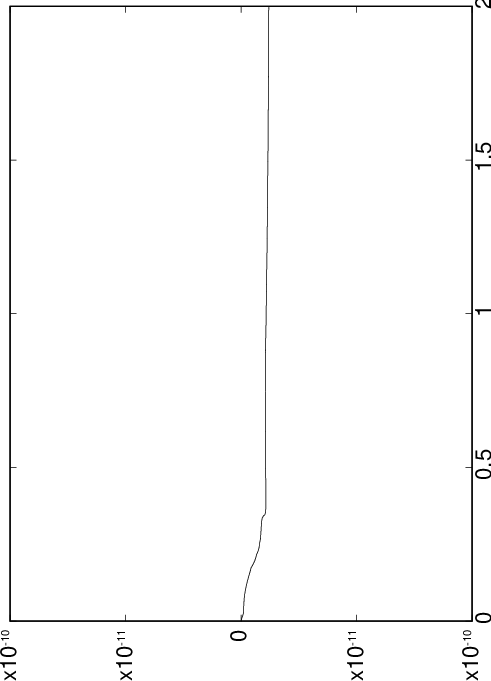}
\caption{$\Gamma^m$ at times $t=0,0.1,0.2,0.5,2$. Below we show a plot of
the discrete energy $|\Gamma^m|$ and of the relative volume loss 
$v_\Delta^m$ over time.}
\label{fig:nonconvex3dfine}
\end{figure}%
We observe that the initially convex interface loses its convexity
during the evolution, which numerically confirms that such evolutions also
exist in the case $d=3$. Recall that the corresponding result for $d=2$
has been shown in \cite{Mayer98}.
For the numerical simulation in Figure~\ref{fig:nonconvex3dfine} 
we also note that
the discrete energy is monotonically decreasing, while the enclosed volume is
maintained up to the chosen solver tolerance.

In a second experiment where an initially convex interface loses its convexity,
we start the evolution with a rounded cylinder of total dimension
$6\times6\times1$. We see from the evolution in 
Figure~\ref{fig:cigar661_K1538} that the moving interface becomes nonconvex,
before it approaches the shape of a sphere.
The discretization parameters for this computation are 
$N_f=128$, $N_c=16$, $\tau=10^{-3}$, $T=2$ and $K=1538$.
\begin{figure}
\center
\includegraphics[angle=-0,width=0.18\textwidth]{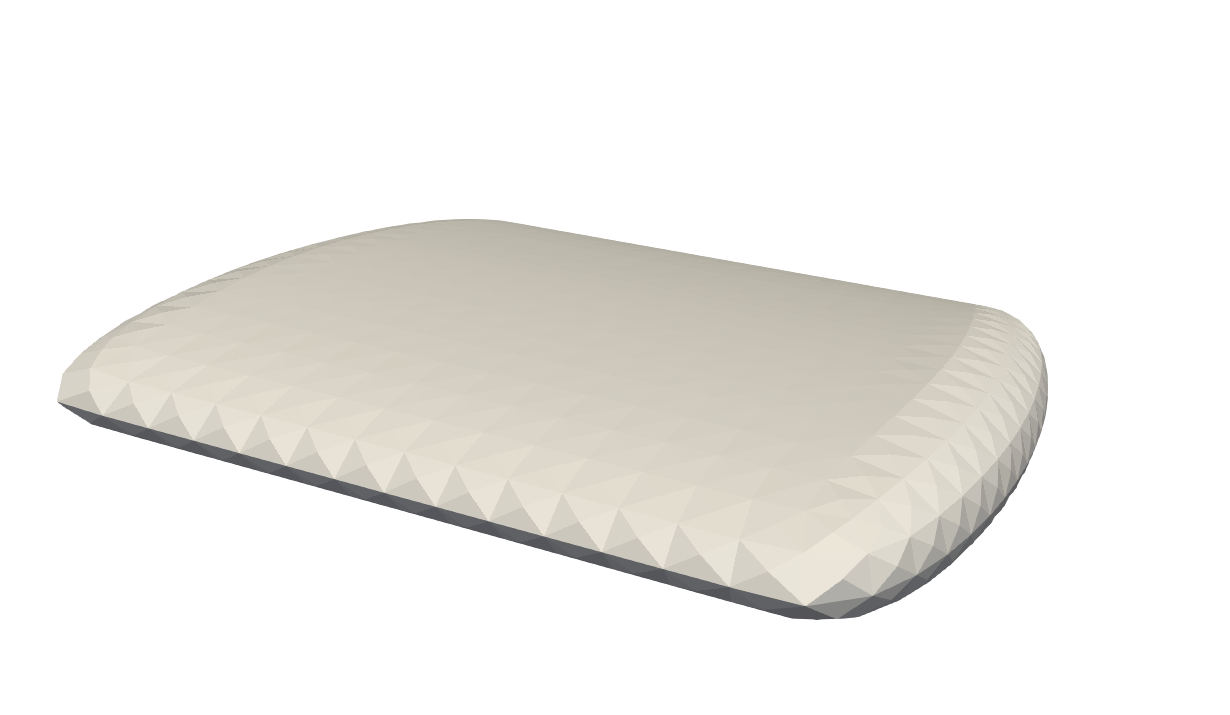}
\includegraphics[angle=-0,width=0.18\textwidth]{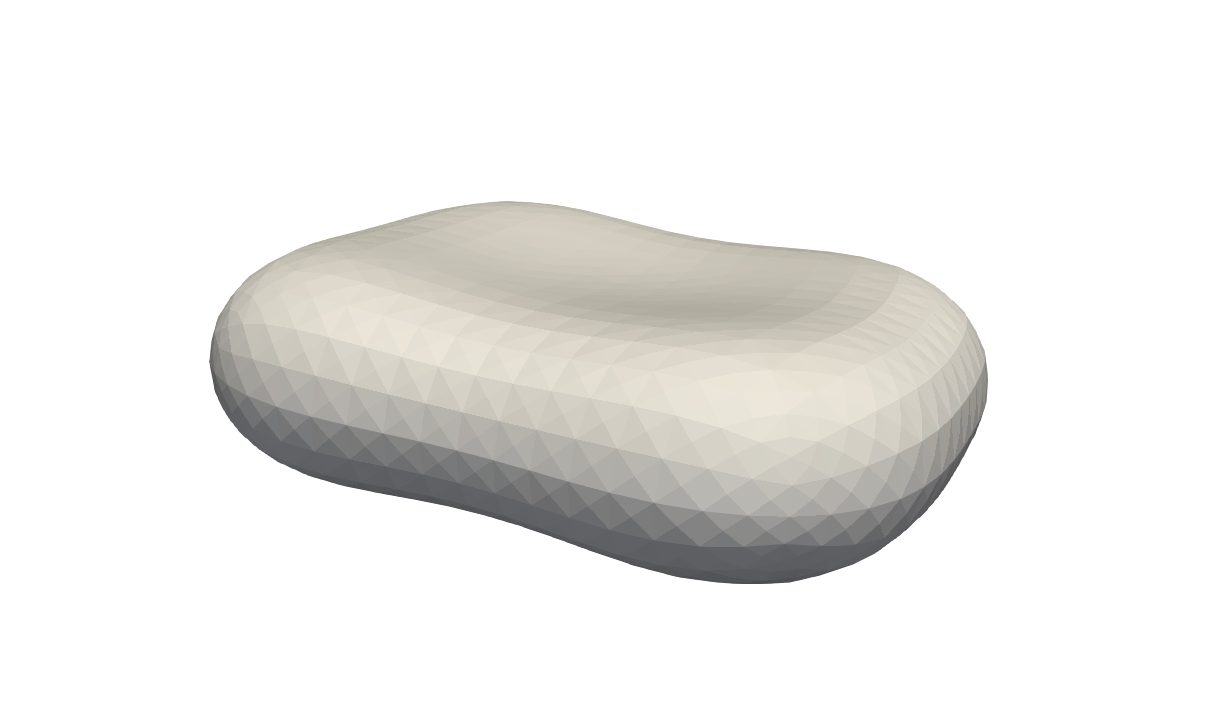}
\includegraphics[angle=-0,width=0.18\textwidth]{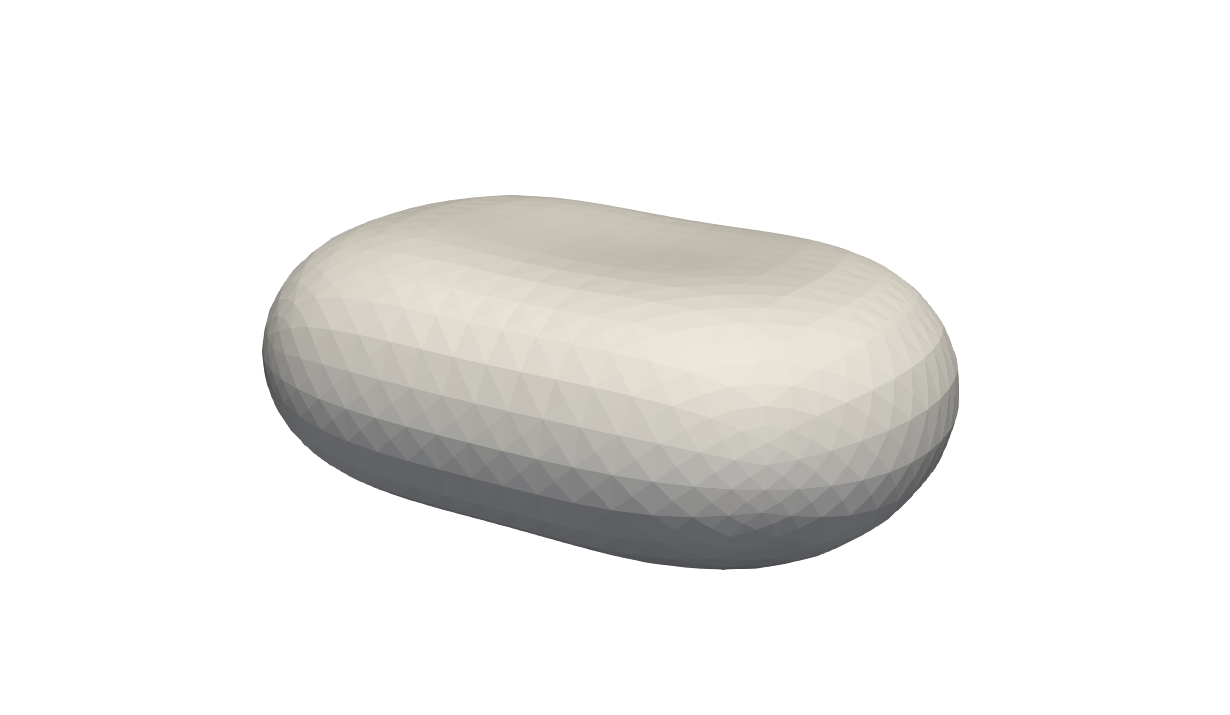}
\includegraphics[angle=-0,width=0.18\textwidth]{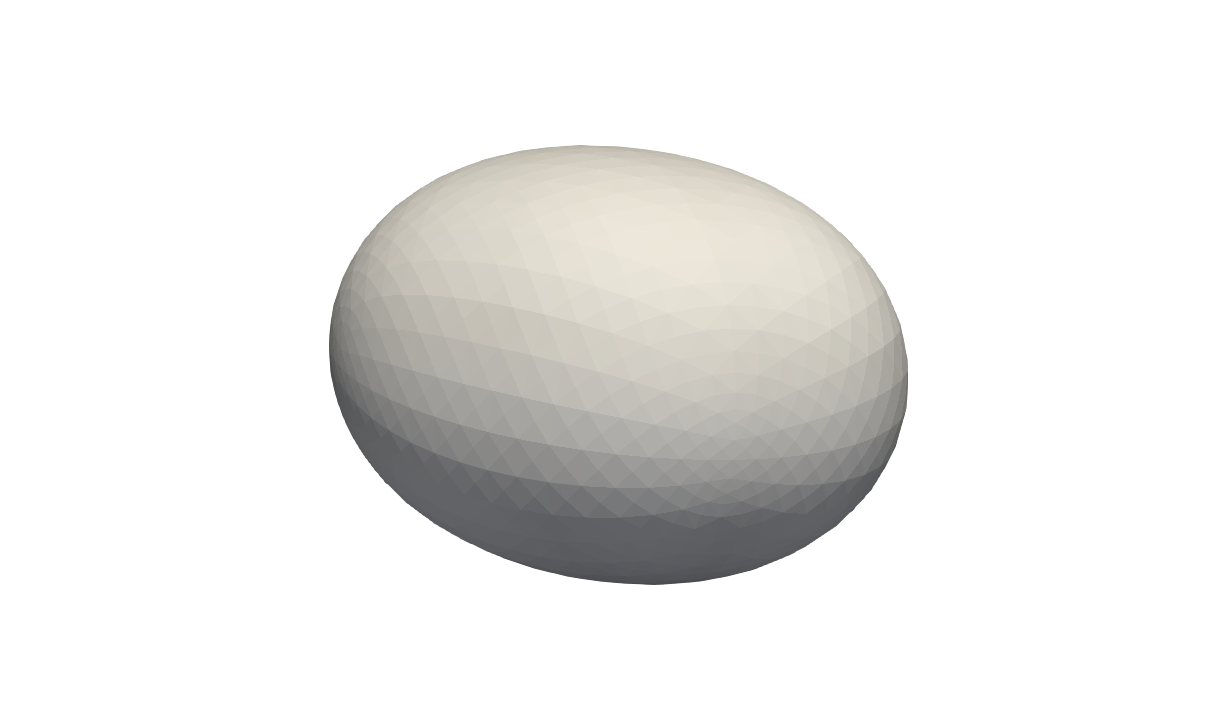}
\includegraphics[angle=-0,width=0.18\textwidth]{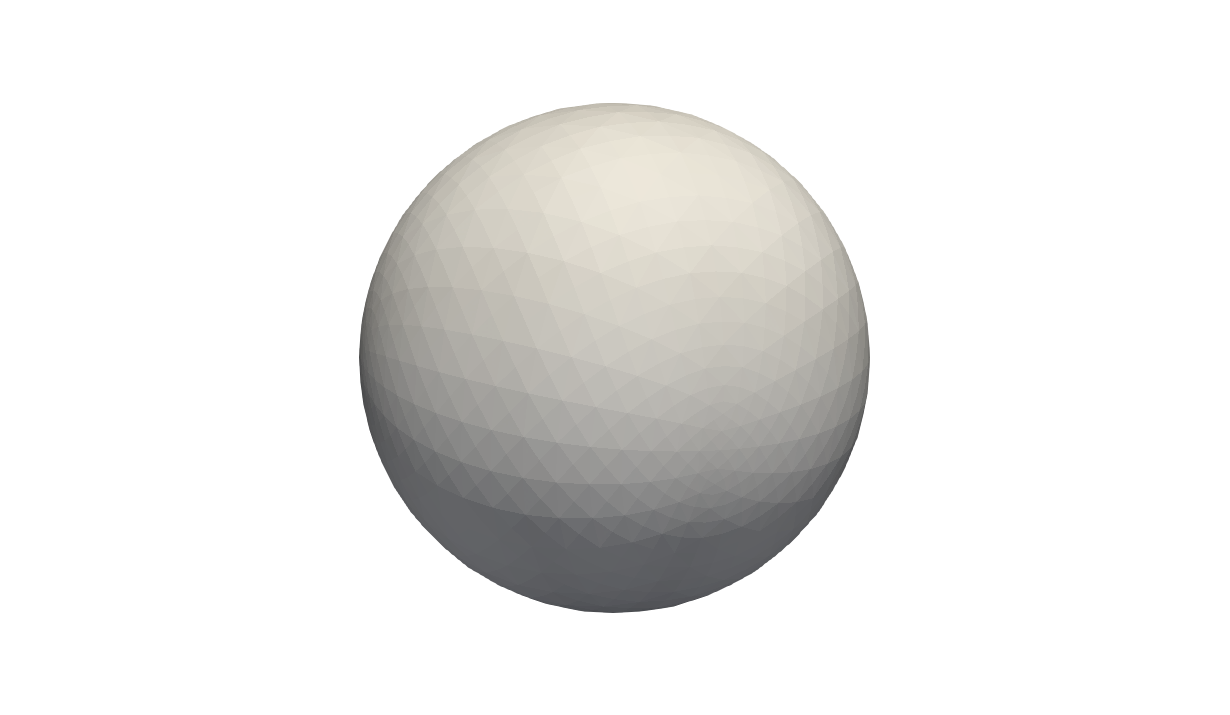}
\\
\includegraphics[angle=-90,width=0.3\textwidth]{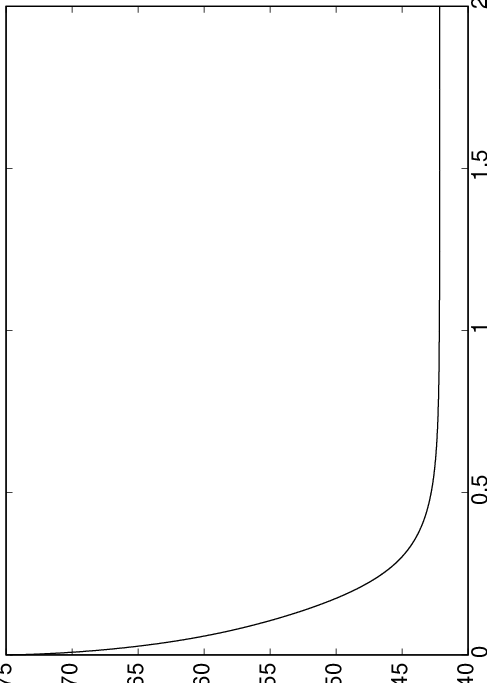}\quad
\includegraphics[angle=-90,width=0.3\textwidth]{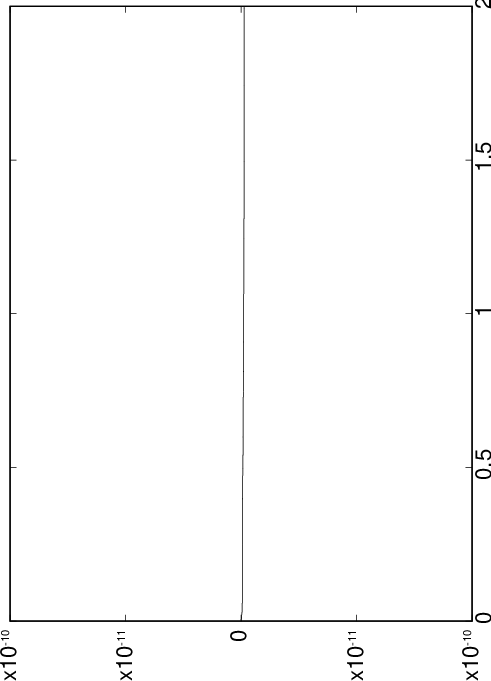}
\caption{$\Gamma^m$ at times $t=0,0.1,0.2,0.5,2$. Below we show a plot of
the discrete energy $|\Gamma^m|$ and of the relative volume loss 
$v_\Delta^m$ over time.}
\label{fig:cigar661_K1538}
\end{figure}%

We also present two simulations for an anisotropic surface energy. 
In the first one, we repeat the simulation in Figure~\ref{fig:nonconvex3dfine},
with the same discretization parameters as before, but now
for the anisotropy
\begin{equation*} 
\gamma(\vec{p}) = \sum_{i=1}^3 \left[ \delta^2|\vec{p}|^2 +
p_i^2(1-\delta^2)\right]^\frac12,\quad \delta = 0.1,
\end{equation*}
which approximates the $\ell^1$--norm of $\vec p$.
For the computation in Figure~\ref{fig:aniL3nonconvex}
it can be observed that, as in the isotropic
case, the interface loses its convexity. Eventually it settles down to an
approximation of the Wulff shape, which here is a smoothed cube.
\begin{figure}
\center
\includegraphics[angle=-0,width=0.18\textwidth]{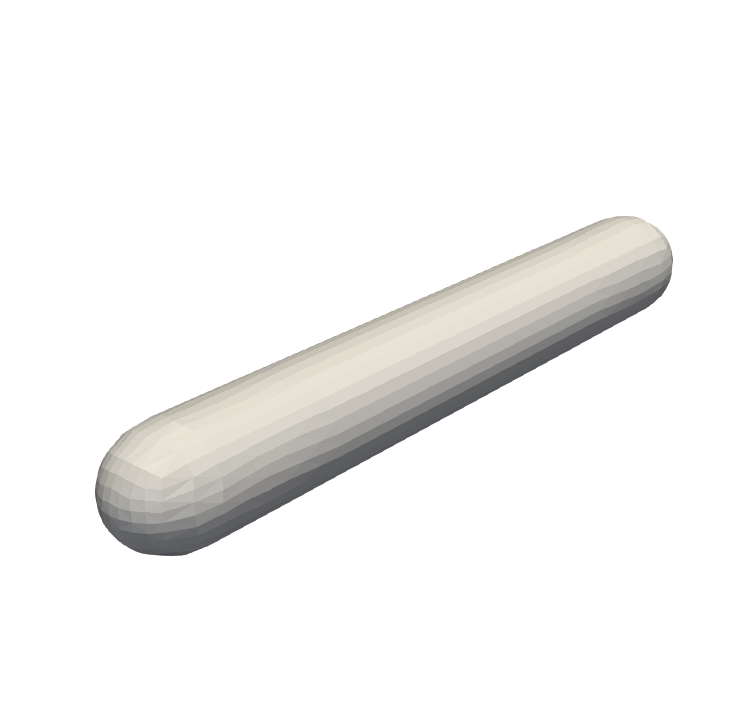}
\includegraphics[angle=-0,width=0.18\textwidth]{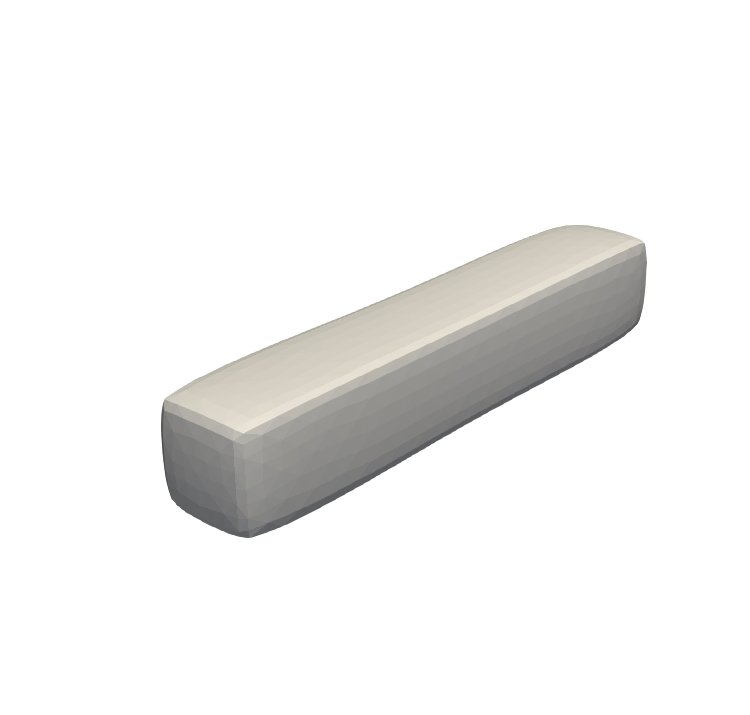}
\includegraphics[angle=-0,width=0.18\textwidth]{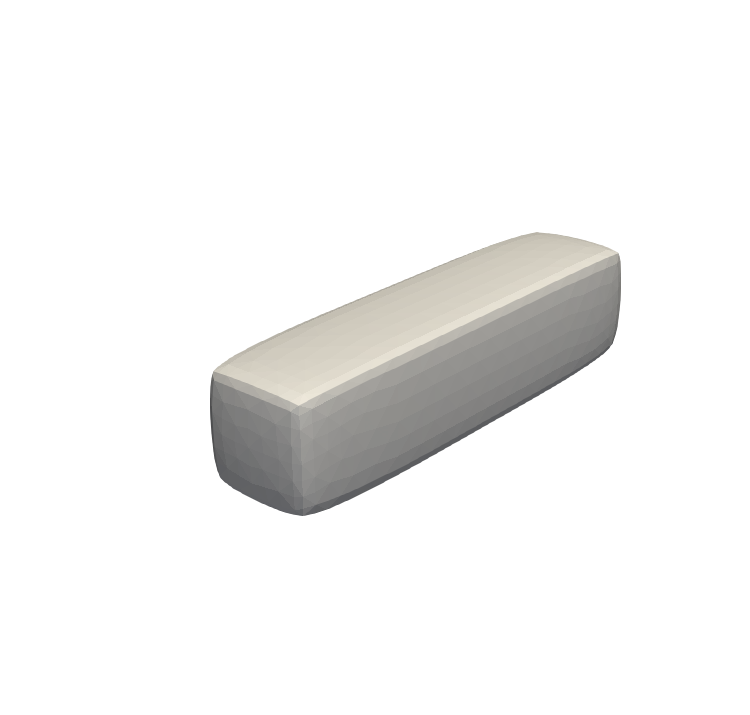}
\includegraphics[angle=-0,width=0.18\textwidth]{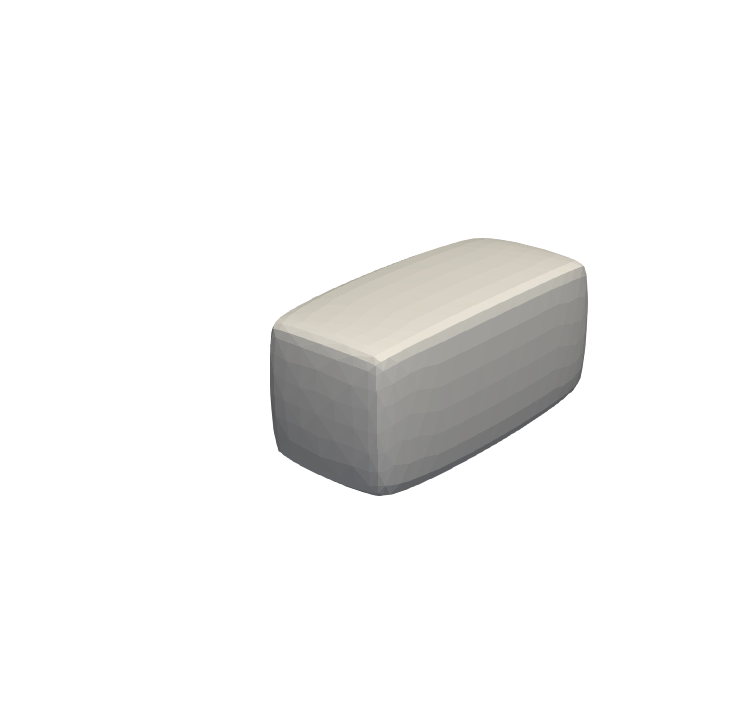}
\includegraphics[angle=-0,width=0.18\textwidth]{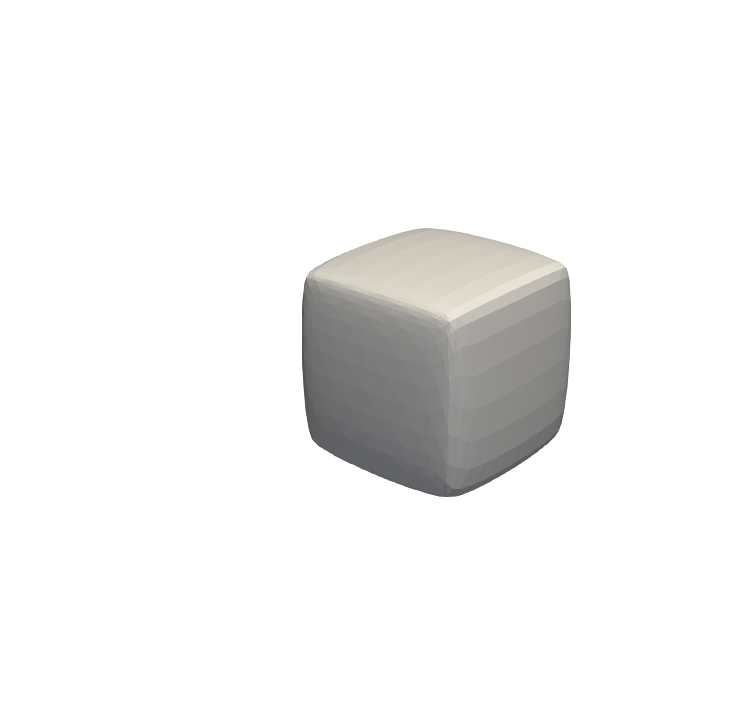}
\\
\includegraphics[angle=-90,width=0.3\textwidth]{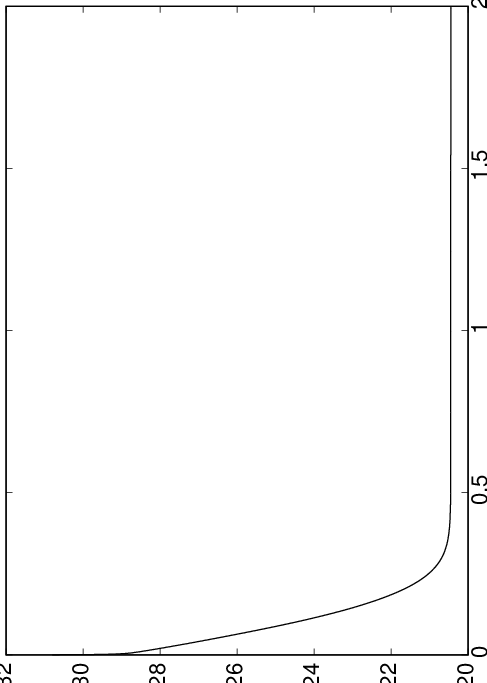}\quad
\includegraphics[angle=-90,width=0.3\textwidth]{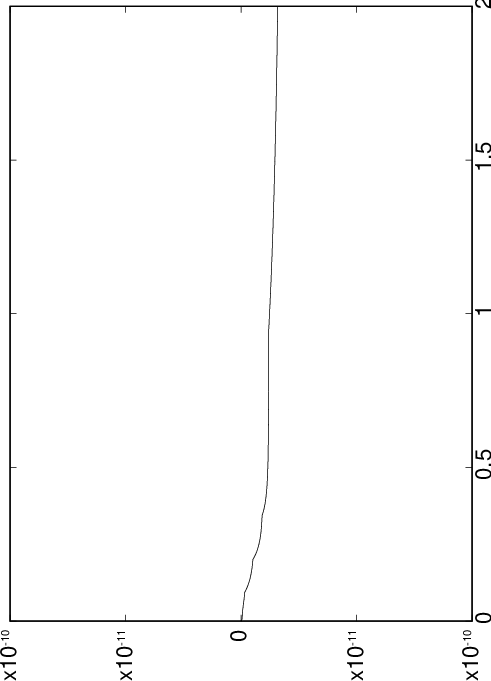}
\caption{$\Gamma^m$ at times $t=0,0.05,0.1,0.2,2$. Below we show a plot of
the discrete energy $|\Gamma^m|_\gamma$ and of the relative volume loss 
$v_\Delta^m$ over time.}
\label{fig:aniL3nonconvex}
\end{figure}%

In the final simulation we use an anisotropic energy of the form 
\eqref{eq:g} with $r>1$, so that the iteration \eqref{eq:itaniMS} also has to
account for the nonlinearity in the approximation of the anisotropy in
\eqref{eq:fdaniMS}. In particular, we choose 
\begin{equation*} 
\gamma(\vec{p}) = \left(\sum_{i=1}^3 \left[ \delta^2|\vec{p}|^2 +
p_i^2(1-\delta^2)\right]^\frac r2 \right)^\frac1r,\quad \delta = 
0.1,\ r = 9,
\end{equation*}
in order to model an anisotropy with an octahedral Wulff shape, see e.g.\
\cite[Figs.\ 4, 15]{ani3d}. For the experiment in 
Figure~\ref{fig:anicigar661} we start from 
the same rounded cylinder of total dimension $6\times6\times1$ from 
Figure~\ref{fig:cigar661_K1538}, and also use the discretization
parameters from the earlier simulation.
During the interesting
evolution the moving interface approaches the Wulff shape, and decreases its
anisotropic surface energy as it does so. As expected, the numerical
approximation conserves the enclosed volume exactly.
\begin{figure}
\center
\includegraphics[angle=-0,width=0.18\textwidth]{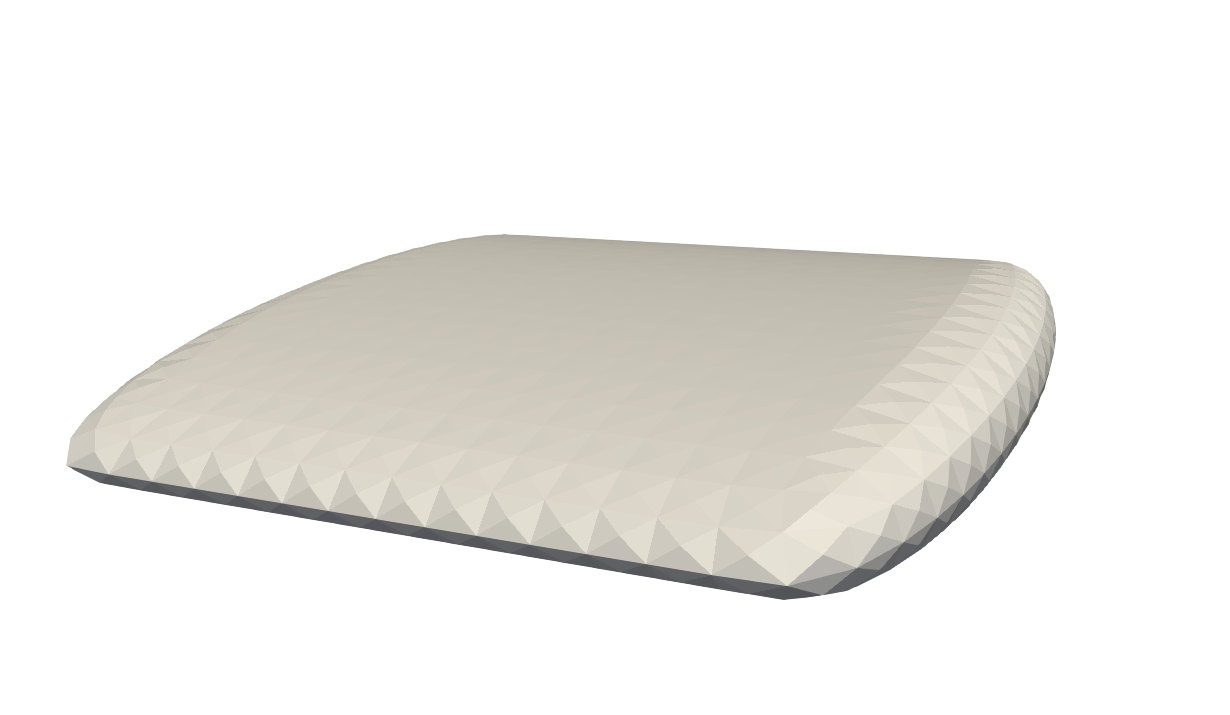}
\includegraphics[angle=-0,width=0.18\textwidth]{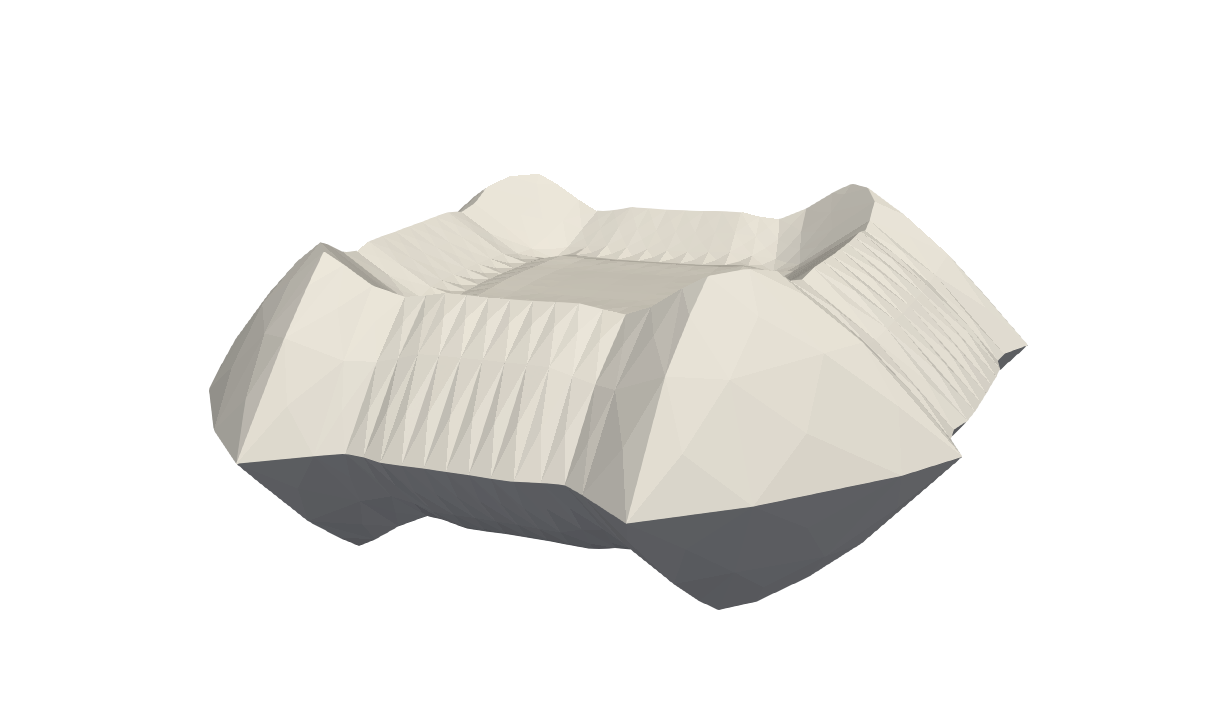}
\includegraphics[angle=-0,width=0.18\textwidth]{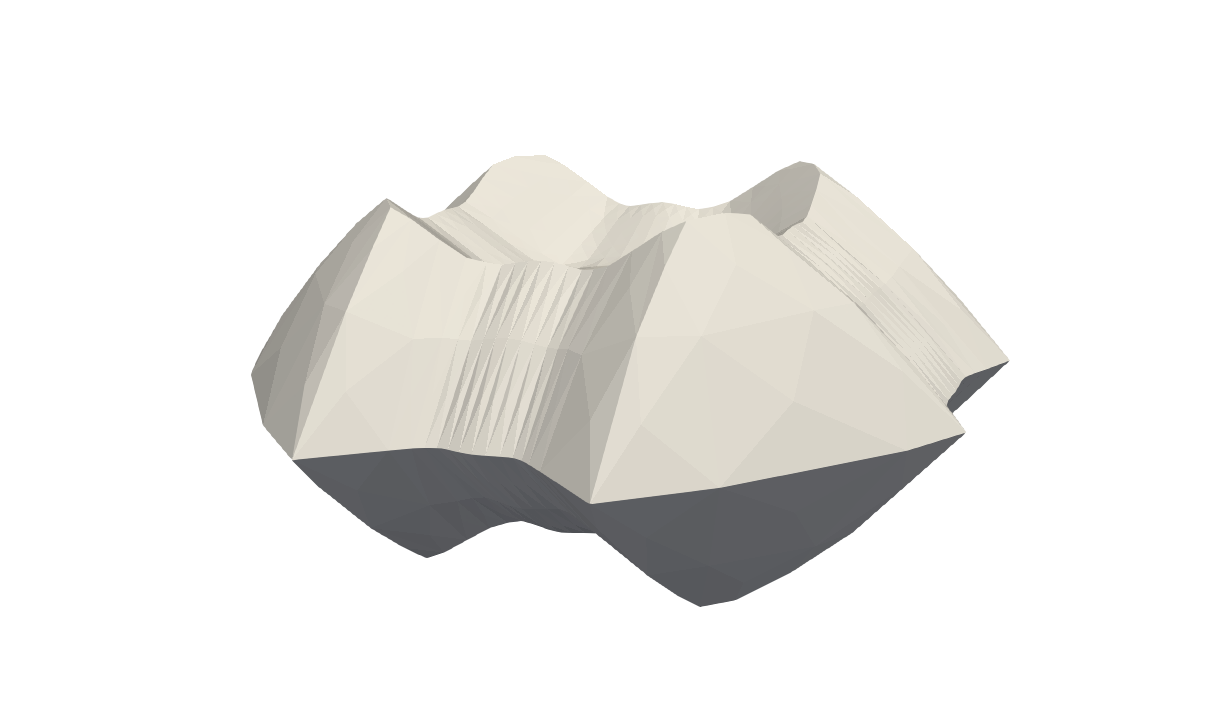}
\includegraphics[angle=-0,width=0.18\textwidth]{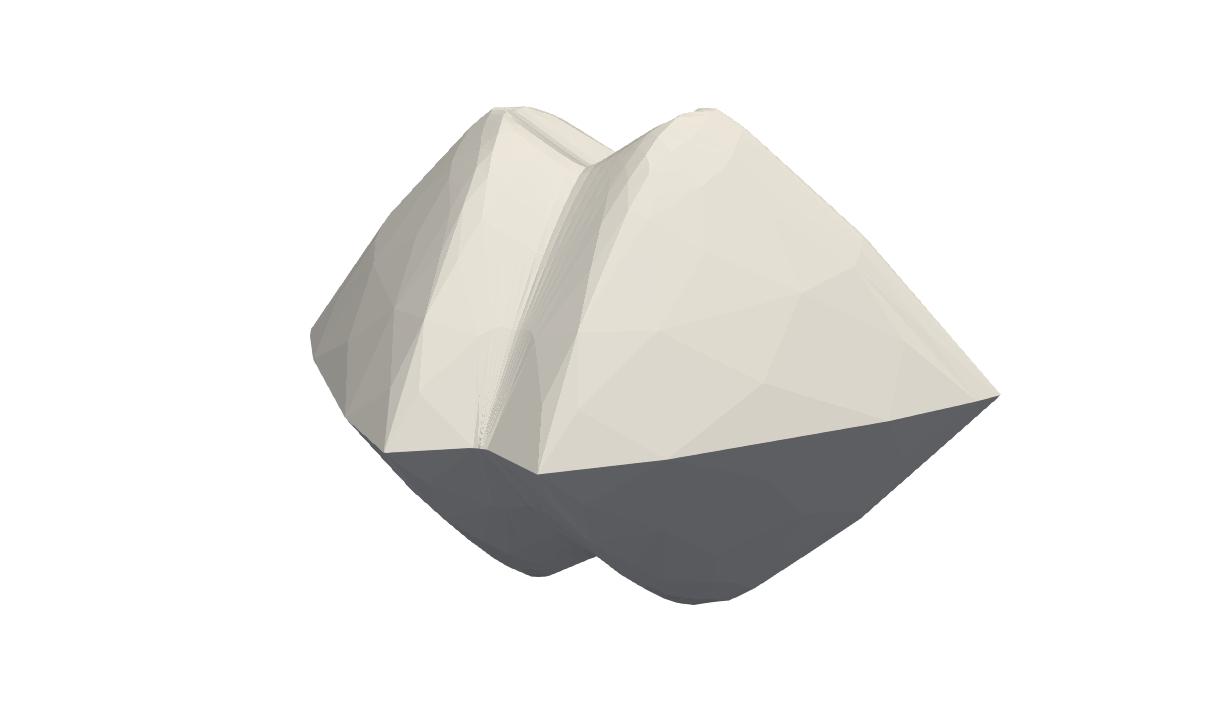}
\includegraphics[angle=-0,width=0.18\textwidth]{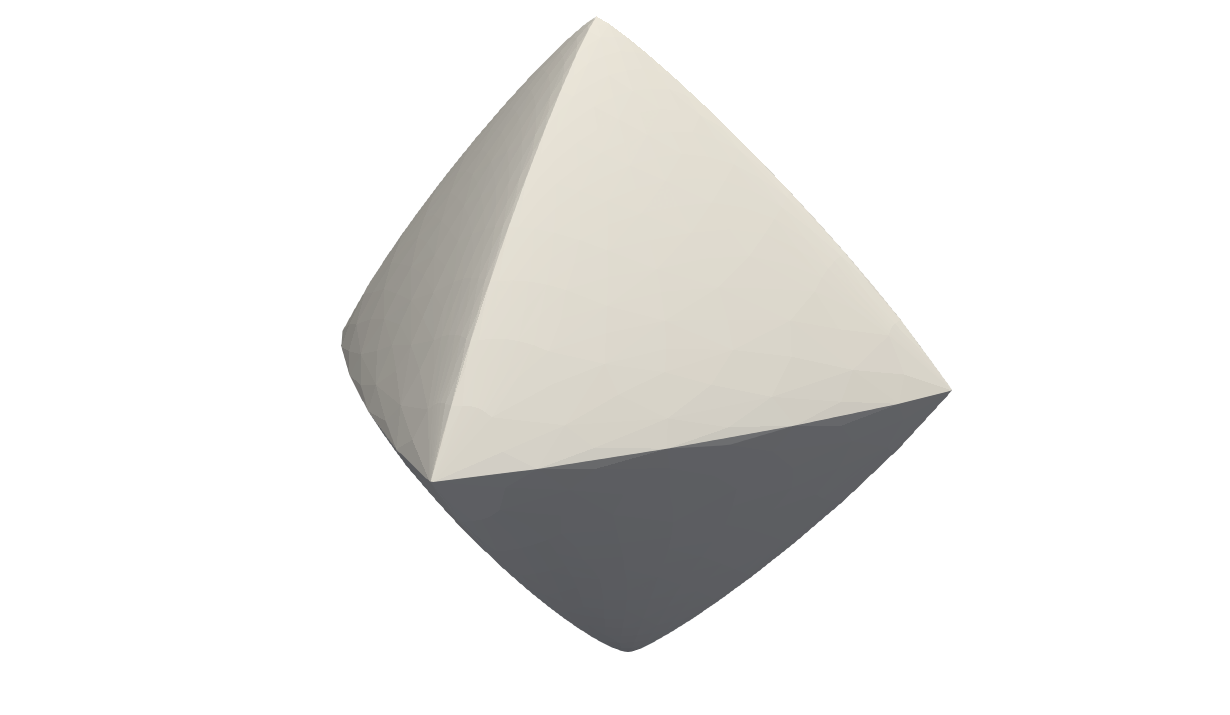}
\\
\includegraphics[angle=-90,width=0.3\textwidth]{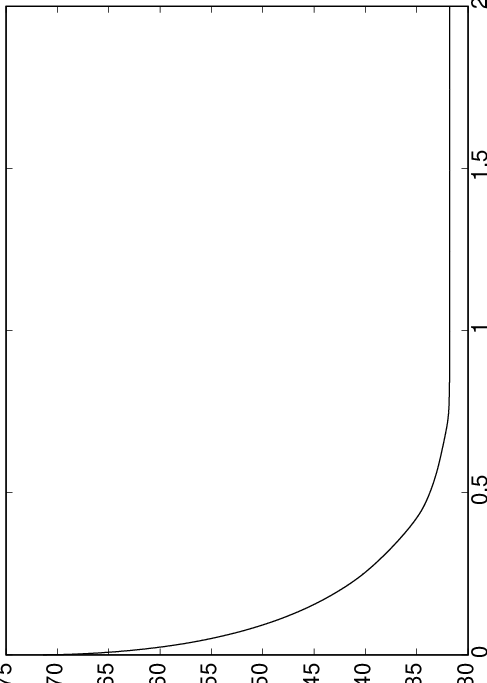}\quad
\includegraphics[angle=-90,width=0.3\textwidth]{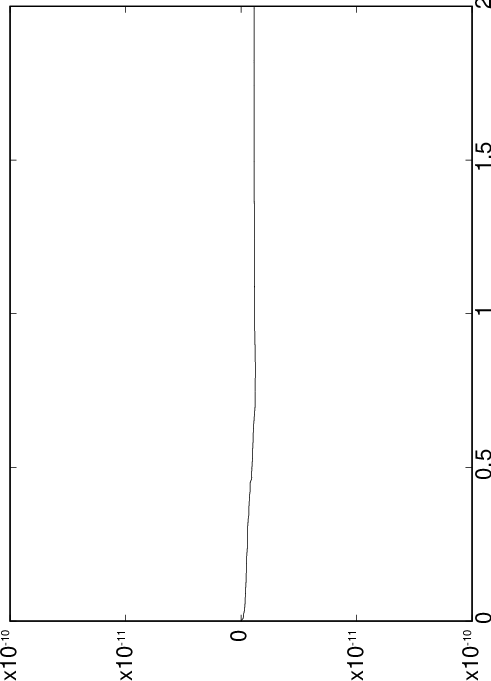}
\caption{$\Gamma^m$ at times $t=0,0.1,0.2,0.5,2$. Below we show a plot of
the discrete energy $|\Gamma^m|$ and of the relative volume loss 
$v_\Delta^m$ over time.}
\label{fig:anicigar661}
\end{figure}%

\def\soft#1{\leavevmode\setbox0=\hbox{h}\dimen7=\ht0\advance \dimen7
  by-1ex\relax\if t#1\relax\rlap{\raise.6\dimen7
  \hbox{\kern.3ex\char'47}}#1\relax\else\if T#1\relax
  \rlap{\raise.5\dimen7\hbox{\kern1.3ex\char'47}}#1\relax \else\if
  d#1\relax\rlap{\raise.5\dimen7\hbox{\kern.9ex \char'47}}#1\relax\else\if
  D#1\relax\rlap{\raise.5\dimen7 \hbox{\kern1.4ex\char'47}}#1\relax\else\if
  l#1\relax \rlap{\raise.5\dimen7\hbox{\kern.4ex\char'47}}#1\relax \else\if
  L#1\relax\rlap{\raise.5\dimen7\hbox{\kern.7ex
  \char'47}}#1\relax\else\message{accent \string\soft \space #1 not
  defined!}#1\relax\fi\fi\fi\fi\fi\fi}

\end{document}